\documentclass[11pt,reqno]{amsart}
\usepackage{amssymb}
\usepackage{amsfonts}
\usepackage{amsthm}
\numberwithin{equation}{section}
\theoremstyle{plain}
\newtheorem{theorem}{Theorem}[section]
\newtheorem{proposition}[theorem]{Proposition}

\theoremstyle{definition}
\newtheorem{definition}[theorem]{Definition}
\newcommand{\refE}[1]{(\ref{E:#1})}
\newcommand{\refS}[1]{Section~\ref{S:#1}}
\newcommand{\refSS}[1]{Section~\ref{SS:#1}}
\newcommand{\refT}[1]{Theorem~\ref{T:#1}}

\newcommand{\refP}[1]{Proposition~\ref{P:#1}}
\newcommand{\refD}[1]{Definition~\ref{D:#1}}

\newcommand{\R}{\ensuremath{\mathbb{R}}}

\newcommand{\im}{\mathrm{\;Im\;}}
\newcommand{\C}{\ensuremath{\mathbb{C}}}
\newcommand{\N}{\ensuremath{\mathbb{N}}}

\renewcommand{\P}{\ensuremath{\mathbb{P}}}
\newcommand{\Z}{\ensuremath{\mathbb{Z}}}

\newcommand{\F}{\ensuremath{\mathbb{F}}}

\newcommand{\A}{\ensuremath{\mathcal{A}}}

\renewcommand{\i}{{\,\mathrm{i}\,}}
\newcommand{\scp}[2]{{\langle #1,#2\rangle}}
\newcommand{\w}{\omega}
\newcommand{\cim}{C^\infty(M)}
\newcommand{\Cim}{C^\infty(M)}
\renewcommand{\d}{\partial}
\newcommand{\db}{\overline{\partial}}
\newcommand{\hh}{\hat h}

\newcommand{\e}{\mathrm{e}}

 \newcommand{\ghm}[1][m]{\Gamma_{hol}(M,L^{#1})}

 \newcommand{\gulm}[1][m]{\Gamma_{\infty}$(M,L^{#1})}
 \renewcommand{\d}{\partial}
 
 \newcommand{\wb}{\overline{w}}
 \newcommand{\zb}{\overline{z}}

  \newcommand{\pnc}[1][N]{\ensuremath{\P^{#1}(\C)}}
 \newcommand{\PGL}{\mathrm{PGL}}
 
 \newcommand{\volu}{\mathrm{vol}}
 \newcommand{\skp}[2]{{\langle #1,#2\rangle}}
 \newcommand{\skps}[3]{{\langle #1,#2\rangle}_{#3}}
 \newcommand{\skpsm}[3]{{\langle #1,#2\rangle}^{(m)}_{#3}}

\def\Lpm{{\mathrm{L}}^2(M,L^m)}

\def\Lqv{{\mathrm{L}}^2(Q,\mu)}

\def\Tfm{T_f^{(m)}}
\def\Tgm{T_g^{(m)}}
\def\Tfgm{T_{\{f,g\}}^{(m)}}
\def\Tma#1{T_{#1}^{(m)}}
\def\Hm{\mathcal {H}^{(m)}}
\def\Hc{{\mathcal {H}}}
\def\d{\partial}

\def\Pfz#1{\frac {\partial #1}{\partial z}}
\def\Pfzb#1{\frac {\partial#1}{\partial\overline{z}}}
\def\End{\mathrm{End}}

\def\Tr{\operatorname{Tr}}
\def\volu{\operatorname{vol}}
\def\hh{\hat h}
\renewcommand{\Im}{I^{(m)}}
\newcommand{\Om}{\Omega}
\newcommand{\Ome}{\Omega^{(m)}_\epsilon}

\newcommand{\eghm}{\mathrm{End}(\Gamma_{hol}(M,L^{(m)}))}
\newcommand{\pimh}{\hat\Pi^{(m)}}
\newcommand{\Bm}{\mathcal{B}_m}
\newcommand{\Pim}{\Pi^{(m)}}

\def\la{\lambda}
\def\a{\alpha}

\def\Cim{C^{\infty}(M)}
\def\ghm{\Gamma_{hol}(M,L^{m})}
\def\ghmo{\Gamma_{hol}(M,L^{m_0})}

\def\gulm{\Gamma_{\infty}(M,L^{m})}
\newcommand{\eam}{e_{\alpha}^{(m)}}
\newcommand{\ebm}{e_{\beta}^{(m)}}
\def\pbar{\overline{\partial}}
\newcommand{\sm}{\sigma^{(m)}}
\newcommand{\svm}{{\check\sigma}^{(m)}}
\newcommand{\epsm}{\epsilon^{(m)}}
\newcommand{\hm}{h^{(m)}}
\begin{document}
%\baselineskip=20pt

%\layout
%%%%%%%%%%%%%%%%%    private header  %%%%%%%%%%%%%%%%%%%%

\title[Berezin-Toeplitz quantization]{Berezin-Toeplitz quantization
\\ for compact K\"ahler manifolds.
\\ A Review of Results}
\author[Martin Schlichenmaier]{Martin Schlichenmaier}
\address[Martin Schlichenmaier]
{University of  Luxembourg, Mathematics Research Unit, FSTC,
6, rue Coudenhove-Kalergi
L-1359 Luxembourg-Kirchberg,
Grand Duchy  of Luxembourg}
 \email{martin.schlichenmaier@uni.lu}
\begin{abstract}
This article is a review on Berezin-Toeplitz operator and
Berezin-Toeplitz
deformation
quantization for compact quantizable K\"ahler manifolds.
The basic objects, concepts, and results are given.
This concerns the correct semi-classical limit behaviour of the
operator quantization, the unique Berezin-Toeplitz deformation 
quantization (star product), covariant and contravariant
Berezin symbols, and Berezin transform.
Other related objects and constructions are also discussed.
\end{abstract}
\subjclass{Primary: 58F06, 53D55; Secondary:  58G15, 53C55, 32C17, 81S10}
%\keywords{deformation quantization, star product, K\"ahler manifolds, 
%Szeg\"o kernel, Berezin transform, coherent states}
\date{December 30, 2009}
\maketitle
\tableofcontents
%%%%%%%%%%%%%%%%%%%%%%%
% section 1
\section{Introduction}\label{S:intro}
%\input intro.tex
%%%%%%%%%%%%   Introduction
%%%%%%%%%%%%%%%%   24.12.09/29.12./30.12.
%%%%%%%%%%%%%%%%%%%%%%%%%%%
For quantizable K\"ahler manifolds the 
Berezin - Toeplitz (BT) quantization scheme, both the 
 operator quantization and the  deformation quantization, supplies 
canonically defined 
quantizations. 
Some time ago, in joint work with Martin Bordemann and
Eckhard Meinrenken, the author of this review showed that
for compact K\"ahler manifolds it is a well-defined
quantization scheme with correct semi-classical limit
\cite{BMS}.

What makes the Berezin-Toeplitz quantization scheme so attractive is
that it does not depend on further choices and that it
does not only produce a formal deformation quantization, but
one which is deeply related to some operator calculus.

{}From the point of view of classical mechanics 
compact K\"ahler manifolds appear as phase space manifolds
of  restricted systems or of reduced systems.
A typical  example of its appearance 
is given by the spherical pendulum which after
reduction has as phase-space  the 
complex projective space.

\medskip

Very recently, inspired by fruitful applications of the basic techniques of 
the Berezin-Toeplitz scheme
beyond the quantization of classical systems, the interest 
in  it revived considerably.

For example these techniques 
show up in non-commutative geometry. More precisely, they appear in the
approach to non-commutative geometry using fuzzy manifolds.
The quantum spaces of the BT quantization of level $m$, defined
in \refS{btq} further down, are finite-dimensional and the
quantum operator of level $m$ constitute finite-dimensional
non-commutative matrix algebras.
This is the arena of non-commutative fuzzy manifolds 
and  gauge theories over them.
The classical limit, the commutative manifold, is obtained as
limit $m\to\infty$.
The name {\it fuzzy sphere}  was coined by John Madore \cite{Mad} 
for a certain quantized version
of the Riemann sphere.
It turned out to be a quite productive direction in the
non-commutative geometry approach to quantum field theory.
It is impossible to give a rather complete list of 
people working in this approach.
The following is a rather erratic and random choice of
references
\cite{Bal1}, \cite{Bal2}, \cite{Sae}, \cite{Dol},
\cite{Asch}, \cite{Gro1},  \cite{Gro2}, \cite{Car}.

\medskip

Another appearance of Berezin-Toeplitz 
quantization techniques as basic ingredients is in the
pioneering work of J\o rgen Andersen on the mapping class group (MCG)
of surfaces in the context of Topological Quantum Field Theory (TQFT).
Beside other results, he was able to proof the asymptotic faithfulness
of the mapping group action on the space of covariantly
constant sections of the Verlinde bundle with respect
to the Axelrod-Witten-de la Pietra and Witten connection
\cite{Andsu,Andab}, see also \cite{Schlsu}.
Furthermore, he showed that the 
MCG does not have Kazhdan's property $T$. 
Roughly speaking, a group has {\it property T} 
says that the identity representation
is isolated in the space of all unitary representations of the group
\cite{Andt}.
In these applications the manifolds to be quantized are the moduli spaces
of certain flat connections on Riemann surfaces or, 
equivalently, the moduli space of stable algebraic vector bundles over
smooth projective curves.
Here further exciting research is going on. 
In particular, in the realm of  TQFT and the construction of
modular functors
\cite{AG}, \cite{AU1,AU2}.

\medskip

In general quite often moduli spaces come with a K\"ahler structure
which is quantizable. Hence, it is not surprising that the  
Berezin-Toeplitz quantization 
scheme is of importance in moduli space problems.
Non-commutative deformations, and a quantization is a
non-commutative
deformation, yield also informations about the commutative
situation.
These aspects clearly need further investigations.

\medskip

There are a lot of other applications on which work has already
been done, recently started, or can be expected.
As the Berezin - Toeplitz scheme has become a basic tool
it seems the right time to collect the techniques and results
in such a review.
We deliberately concentrate on the case of compact 
K\"ahler manifolds. In particular, we stress the methods
and results valid for all of them.
Due to ``space-time'' limitations we will not deal with
the non-compact situation. In this situation case by case
studies of the examples, or class of examples are needed.
See \refSS{noncompact}
for references to some of them.
Also we have to skip presenting recent attempts to deal with special
singular situations, like orbifolds, but see at least
\cite{MaMar,MaMar1}, \cite{Charl1}.

\medskip

Of course, there are other reviews presenting similar quantization schemes.
A very incomplete list  is the following
\cite{AD}, \cite{AE}, \cite{Ali}, \cite{Stern}, \cite{DiSt}.

\medskip

This review is self-contained in the following sense.
I try to explain all notions and concepts needed to understand the
results and theorems only assuming some  background in modern
geometry and analysis.
And as such it should be accessible for a newcomer to the 
field (both for mathematicians as for physicists)
and help him to enter these interesting research directions.
It is not self-contained in the strict sense as it does 
supply only those proofs or sketches of proofs which are
either not available elsewhere or are essential for the understanding
of the statements and concepts.
The review does not require a background in  quantum physics as 
only mathematical aspects of quantizations are touched on.

%\newpage
%%%%%%%%%%%%%%%%%%%%%%%%%%%%%%%%%%%
\section{The set-up of geometric quantization}\label{S:set}
%\input setup.tex
%%%%     The set up of geometric quantization
%%%      29.12.09  / 30.12.09
%%%%%%%%%%%%%%%%%%%%%%%%%%%%%
In the following I will recall the principal set-up of geometric
quantization which is usually done for symplectic manifolds in the
case when the manifold is a K\"ahler manifold.

\subsection{K\"ahler manifolds}
We will only consider  phase-space manifolds which carry the
structure of a 
K\"ahler manifold $\ (M,\w)$.
Recall that in this case $M$ is a complex manifold and 
$\w$, the K\"ahler form, is
a non-degenerate closed positive $(1,1)$-form.

If the complex dimension of $M$ is $n$ then the K\"ahler form $\w$ can be 
written with respect to local holomorphic coordinates
${\{z_i\}}_{i=1,\ldots,n}$ as 
\begin{equation}
\w=\i\sum_{i,j=1}^{n} g_{ij}(z)dz_i\wedge d\bar z_j,
\end{equation}
with local functions $g_{ij}(z)$ such that the 
matrix $(g_{ij}(z))_{i,j=1,\ldots,n}$ is hermitian and
positive definite.

Later on we will assume that $M$ is a compact K\"ahler manifold.

\subsection{Poisson algebra}
Denote by $\Cim$ the algebra of 
complex-valued (arbitrary often) differentiable functions
with the point-wise multiplication as associative product.
A symplectic form on a differentiable manifold is a closed
non-degenerate
2-form. In particular, 
we can consider our  K\"ahler form $\w$ as a symplectic form.

For symplectic manifolds we can introduce on 
$\Cim$ a Lie algebra structure, the Poisson bracket
{\em Poisson bracket} $\{.,.\}$, in the following way.
First we a
assign to every  $f\in\Cim$ its {\em Hamiltonian vector field} $X_f$, and 
then to every pair of functions $f$ and $g$ the 
{\em Poisson bracket} $\{.,.\}$ via
\begin{equation}
\label{E:Poi}
\w(X_f,\cdot)=df(\cdot),\qquad  
\{\,f,g\,\}:=\w(X_f,X_g)\ .
\end{equation}
One verifies that this is indeed a Lie algebra and that furthermore
we have the Leibniz rule
$$
\{fg,h\}=f\{g,h\}+\{f,h\}g,\qquad  \forall f,g,h\in\Cim.
$$
Such a compatible structure is called a
{\em Poisson algebra}.

\subsection{Quantum line bundles}
A {\em quantum line bundle} for a given symplectic manifold $(M,\w)$ is 
a triple $(L,h,\nabla)$, where $L$ is a complex line bundle, 
$h$ a Hermitian metric   on $L$, and
 $\nabla$ a connection  compatible with the metric $h$
such that the (pre)quantum condition 
\begin{equation}\label{E:quant}
\begin{gathered}
\mathrm{curv}_{L,\nabla}(X,Y):=\nabla_X\nabla_Y-\nabla_Y\nabla_X-\nabla_{[X,Y]}
=
-\i\w(X,Y),\\
\text{ resp.} \quad \mathrm{curv}_{L,\nabla}=-\i\w \ 
\end{gathered}
\end{equation}
is fulfilled.
A symplectic manifold is called {\em quantizable} if there exists
a quantum line bundle.

In the situation of K\"ahler manifolds we require for a quantum 
line bundle that it is holomorphic and that the connection is
compatible both with the metric $h$ and the complex structure of
the bundle.
In fact, by this requirement $\nabla$ 
will be uniquely fixed.
If we choose local holomorphic  coordinates
on the manifold and  and a 
local holomorphic frame of the bundle
the metric $h$ will be  
represented by 
a function $\hh$.
In this case the  curvature in the bundle can be given 
by $\db\d\log\hh$ and the quantum condition reads as  
\begin{equation}\i\db\d\log\hh=\w\ .
\end{equation}

\subsection{Example: The Riemann sphere} 
The Riemann sphere is  
the complex projective line
 $\P^1(\C)=\C\cup \{\infty\}\cong S^2$.
 With respect to the quasi-global
coordinate $z$  the form can be given as
\begin{equation}
\w=\frac {\i}{(1+z\zb)^2}dz\wedge d\zb\ .
\end{equation}
For the Poisson bracket one obtains
\begin{equation}
\{f,g\}=\i(1+z\zb)^2\left(\Pfzb f\cdot\Pfz g-\Pfz f\Pfzb g\right)\ .
\end{equation}
Recall that the points in $\P^1(\C)$ correspond to lines in $\C^2$
passing through the origin.
If we assign to every point in $\P^1(\C)$ the line it represents we
obtain a holomorphic line bundle, called the tautological 
line bundle. The hyper plane section bundle is dual to the
tautological
bundle. It turns out that it is a quantum line bundle.
Hence $\P^1(\C)$ is quantizable.

\subsection{Example: The complex projective space}
Next we consider  the $n$-dimensional  complex projective space
 $\P^n(\C)$.
The example above can be extended to the projective
space of any dimension.
The K\"ahler form is given by the Fubini-Study form
\begin{equation}
\w_{FS}:=
\i\frac 
{(1+|w|^2)\sum_{i=1}^ndw_i\wedge
d\wb_i-\sum_{i,j=1}^n\wb_iw_jdw_i\wedge d\wb_j} {{(1+|w|^2)}^2}
\ .
\end{equation}
The coordinates 
$w_j$, $j=1,\ldots, n$ are affine coordinates
$w_j=z_j/z_0$ on the affine chart 
\newline
$U_0:=\{(z_0:z_1:\cdots:z_n)\mid z_0\ne 0\}$.
Again, $\P^n(\C)$ is quantizable with the hyper plane section bundle
as quantum line bundle.

\subsection{Example: The torus}
The (complex-) one-dimensional torus
can be given as
$M = \C/\Gamma_\tau$ where $\ \Gamma_\tau:=\{n+m\tau\mid n,m\in\Z\}$
is a lattice with $\ \im \tau>0$.
As K\"ahler form we take
\begin{equation}
\w=\frac {\i\pi}{\im \tau}dz\wedge d\zb\ ,
\end{equation}
with respect to the coordinate $z$ on the
covering space $\C$.
Clearly this form is invariant under 
the lattice $\Gamma_\tau$ and hence
well-defined on $M$.
For the Poisson bracket one obtains
\begin{equation}
\{f,g\}=\i\frac{\im\tau}{\pi}\left(\Pfzb f\cdot\Pfz g-\Pfz f\Pfzb g\right)\ .
\end{equation}
The corresponding quantum line bundle is the theta line bundle
of degree 1, i.e. the bundle whose global sections are
scalar multiples of the Riemann theta function.

\subsection{Example: The unit disc and Riemann surfaces}
The unit disc 
\begin{equation}
\mathcal{D}:=\{z\in\C\mid  |z|<1\}
\end{equation} 
is a non-compact K\"ahler manifold. 
The K\"ahler form is given by
\begin{equation}\label{E:kfud}
\w=\frac {2\i}{(1-z\zb)^2}dz\wedge d\zb\ .
\end{equation}
Every compact Riemann surface $M$ of genus $g\ge 2$ can be
given 
as the quotient of the  unit disc 
 under the fractional linear transformations of
a Fuchsian subgroup of $SU(1,1)$.
If $\ R=
\begin{pmatrix}
 a&b\\ \overline{b} &\overline{a}
\end{pmatrix} \ $
with $\  |a|^2-|b|^2=1\ $ (as an element of $SU(1,1)$) then
the action is
\begin{equation}
z\mapsto R(z):=\frac {a z + b} {\overline{b} z + \overline{a}}\ .
\end{equation}
The K\"ahler form \refE{kfud}
is invariant under the fractional linear
transformations. Hence it defines a K\"ahler form on $M$.
The quantum bundle is the canonical bundle, i.e. the bundle
whose local sections are the holomorphic differentials.
Its global sections can be identified with the automorphic forms
of weight $2$ with respect to the Fuchsian group.

\subsection{Consequences of quantizability}
The above examples might create the wrong impression that every
K\"ahler manifold is quantizable. This is not the case. For example
only 
those higher dimensional tori complex tori are quantizable which
are abelian varieties, i.e. which admit enough theta functions.
It is well-known that for $n\ge 2$ a generic torus will not
be an abelian variety.
Why this implies that they will not be quantizable we will see in
a moment.

In the language of differential geometry a line bundle is called
 a positive line bundle if its curvature form (up to a factor
of $1/\mathrm{i}$)  is a positive form.
As the K\"ahler form is positive the quantum condition
\refE{quant} yields that a quantum line bundle  $L$
is a positive line bundle.

\subsection{Embedding into projective space}
\label{SS:embedd}
In the following we assume that $M$ is a quantizable compact 
K\"ahler manifold
with quantum line bundle $L$.
Kodaira's embedding theorem   says that $L$ 
is  ample, i.e. that there exists  a certain
tensor power $L^{m_0}$ of $L$ 
such that 
the global holomorphic sections of  $L^{m_0}$ 
can be used to embed the phase space manifold $M$ 
into the projective
space of suitable dimension.
The embedding is defined as follows.
Let $\ghmo$ be the vector space of global holomorphic 
sections of the bundle $L^{m_0}$. Fix  a basis 
$s_0,s_1,\ldots,s_N$.
We choose  local holomorphic coordinates $z$ for $M$ and
 a local holomorphic frame $e(z)$
for the bundle $L$.
After these choices the basis elements 
can be uniquely described by
local holomorphic functions
$\hat s_0,\hat s_1,\ldots,\hat s_N$ defined via $s_j(z)=\hat s_j(z)e(z)$.
The embedding is given by the map
\begin{equation}\label{E:embed}
M\hookrightarrow \pnc,\quad z\mapsto \phi(z)=
(\hat s_0(z):\hat s_1(z):\cdots:\hat s_N(z))\ .
\end{equation}
Note that the point $\phi(z)$ 
in projective space neither depends on the choice of
local coordinates nor on the choice of the local frame for the bundle $L$.
Furthermore 
a different choice of basis 
correspond to
a $\PGL(N,\C)$ action
on the embedding space and hence the embeddings
are projectively equivalent.

By this embedding quantizable compact K\"ahler manifolds are 
complex submanifolds of projective spaces.
By Chow's theorem \cite{SchlRS} they can be given as zero sets of
homogenous polynomials, i.e. they are smooth projective
varieties.
The converse is also true.
Given a smooth subvariety $M$ of $\P^n(\C)$ it will become a
K\"ahler manifold by restricting the Fubini-Study form.
The restriction of the hyper plane section bundle will be
an associated quantum line bundle.

At this place a warning is necessary.
the embedding is only an embedding as complex manifolds
not an isometric embedding as K\"ahler manifolds.
This means that in general 
$\phi^{-1}(\w_{FS})\ne \w$. See \refSS{pfsm} for results on
an ``asymptotic expansion'' of the  pullback.

A line bundles whose global holomorphic sections will define
an embedding into projective space, 
is called a {\it very ample line bundle}.
In the following we will assume that $L$ is already   very ample.
If $L$ is not very ample we choose  $m_0\in\N$  such that 
the bundle $L^{m_0}$ is very ample and take  
this bundle as quantum line bundle with respect to 
the rescaled K\"ahler form $m_0\,\w$  
on $M$. The   underlying complex manifold structure
will not change.

%%%%%%%%%%%%%%%%%%%%%%%%%%%%%%%%%%
\section{Berezin-Toeplitz operators}
\label{S:btq}
%\input btq.tex
%  Berezin- Toeplitz operators
%%%%%%%%%%%%%%%%%%%
%   25.12.2009/ 29.12.2009/30.12.
%%%%%%%%%%%%%%%%%%%%%%%%%%%%%%%%%%%%%%%

In this section we will consider an operator quantization. 
This says that we will assign to each differentiable%
\footnote{differentiable will always mean differentiable to any 
order}
function $f$ on our K\"ahler manifold $M$  (i.e. on our ``phase
space'')
the Berezin-Toeplitz (BT) quantum operator $T_f$. 
More precisely, we will consider a whole family of operators $\Tfm$.
These operators are defined in a canonical way. As we know from
the Groenewold-van Howe theorem
we cannot  expect that the Poisson bracket 
on $M$ can be represented by the Lie algebra of operators if 
we require certain desirable conditions, see \cite{AM}
for further details.
The best we can expect is that we obtain it at least
``asymptotically''.
In fact, this is true.

In our context also the operator of geometric 
quantization exists.
At the end of this section we will discuss its relation to
the BT quantum operator. It will turn out that
if we take for the geometric quantization 
the K\"ahler polarization then they have the same
asymptotic behaviour.

%%%%%%%%%%%%%%%%%%%%%%%%%

\subsection{Tensor powers of the quantum line bundle}
Let $(M,\w)$ be a compact quantizable K\"ahler manifold and $(L,h,\nabla)$ a 
quantum line bundle. We assume  that $L$  is 
already very ample. We consider all its tensor powers
\begin{equation}
(L^m,h^{(m)},\nabla^{(m)}).
\end{equation}
Here  $L^m:=L^{\otimes m}$.
If
 $\hat h$ corresponds to the metric $h$ with respect to 
 a local holomorphic frame $e$
of the bundle $L$ then $\hat h^m$ corresponds to the metric $h^{(m)}$ 
with respect to the frame $e^{\otimes m}$ for the bundle $L^m$.
The connection $\nabla^{(m)}$ will be the induced connection.

We 
introduce a scalar product on the space of sections. 
In this review we adopt the convention that a hermitian
metric (and a scalar product) is anti-linear in the first
argument and linear in the second argument.
First we take the Liouville form
$\ \Omega=\frac 1{n!}\w^{\wedge n}\ $ as volume form  on $M$
and then set 
for  the scalar product
and the norm
\begin{equation}
\label{E:skp}
\langle\varphi,\psi\rangle:=\int_M h^{(m)} (\varphi,\psi)\;\Omega\  ,
\qquad
||\varphi||:=\sqrt{\langle \varphi,\varphi\rangle}\ ,
\end{equation}
on the space $\gulm$ of global $C^\infty$-sections.
Let $\Lpm$ be the  L${}^2$-completion of $\gulm$,  and
$\ghm$ be its (due to the compactness of $M$) 
finite-dimensional closed subspace of global holomorphic
sections.
Let 
\begin{equation}
\ \Pim:\Lpm\to\ghm\ 
\end{equation}
 be the projection.
\begin{definition}
For $f\in\Cim$ the {\it Toeplitz operator  $\Tfm$ (of level $m$)}
is defined by
\begin{equation}
\Tfm:=\Pim\, (f\cdot):\quad\ghm\to\ghm\ .
\end{equation}
\end{definition}
\noindent
In words: One takes a holomorphic section $s$ 
and multiplies it with the
differentiable function $f$. 
The resulting section $f\cdot s$ will only be  differentiable.
To obtain a holomorphic section, one has to project  it back
on the subspace of holomorphic sections.

The linear map
\begin{equation}
\Tma {}:\Cim\to \End\big(\ghm\big),\qquad  f\to \Tma f=\Pi^{(m)}(f\cdot)
\ , m\in\N_0\ .
\end{equation}
is the  {\it Toeplitz}  or {
\it Berezin-Toeplitz quantization map (of level $m$)}.
It will 
neither be a Lie algebra homomorphism nor
an associative algebra homomorphism as
in general
$$
T^{(m)}_f\, T^{(m)}_g=\Pi^{(m)}\,(f\cdot)\,\Pi^{(m)}\,(g\cdot)\,\Pi^{(m)}\ne
\Pi^{(m)}\,(fg\cdot)\,\Pi =T^{(m)}_{fg}.
$$

Furthermore, 
on a fixed  level $m$ it
is  a map
from the infinite-dimensional 
commutative algebra of functions to a noncommutative
finite-dimensional (matrix) algebra.
The finite-dimensionality is 
due to the compactness of $M$.
A lot of classical information will get lost.
 To recover this
information one has to  consider not just a single level $m$ but
all levels together.
\begin{definition}
The Berezin-Toeplitz quantization is the map
\begin{equation}
\Cim\to\prod_{m\in\N_0}\eghm,\qquad
f\to(\Tfm)_{m\in\N_0}.
\end{equation}
\end{definition}
We obtain a family of
finite-dimensional(matrix) algebras and a family of maps.
This infinite family should in some sense ``approximate'' the
algebra $\Cim$.
%%%%%%%%%%%%%%%%%%%%%%%%%%%%%%%%%%

\subsection{Approximation results}
%%%%%%%%%%%%%%%%%%%%%%%%%%%%%%%%%%%
Denote for $f\in\Cim$ by $|f|_\infty$ the sup-norm of $\ f\ $ on $M$ and by 
\begin{equation}
||\Tfm||:=\sup_{\substack {s\in\ghm\\ s\ne 0}}\frac {||\Tfm s||}{||s||}
\end{equation}
the operator norm with respect to the norm \refE{skp}
on $\ghm$.
The following theorem was shown in 1994.
\begin{theorem}
\label{T:approx}
[Bordemann, Meinrenken, Schlichenmaier]
\cite{BMS}

\noindent
(a) For every  $\ f\in \Cim\ $ there exists a $C>0$ such that   
\begin{equation}\label{E:norma}
|f|_\infty -\frac Cm\quad
\le\quad||\Tfm||\quad\le\quad |f|_\infty\ .
\end{equation}
In particular, $\lim_{m\to\infty}||\Tfm||= |f|_\infty$.

\noindent
(b) For every  $f,g\in \Cim\ $ 
\begin{equation}
\label{E:dirac}
||m\i[\Tfm,\Tgm]-\Tfgm||\quad=\quad O(\frac 1m)\ .
\end{equation}

\noindent
(c) For every  $f,g\in \Cim\ $ 
\begin{equation}\label{E:prod}
||\Tfm\Tgm-T^{(m)}_{f\cdot g}||\quad=\quad O(\frac 1m)
\ .
\end{equation}
\end{theorem}
\noindent
%%%%%%%%%%%%%%%%%%%%%%%%%
These results are contained in  Theorem  4.1, 4.2,
and in Section 5 in \cite{BMS}.
We will indicate the proof for (b) and (c) in \refS{global}.
It will make reference to the symbol calculus of generalised 
Toeplitz operators as developed by Boutet-de-Monvel and Guillemin
\cite{BGTo}.
The original proof of (a) was quite involved and required 
Hermite distributions and related objects.
On the basis of the asymptotic expansion of the Berezin transform
\cite{KS} a more direct proof can be given. I will discuss this
in \refSS{normp}.

Only on the basis of this  theorem 
we are allowed to call our scheme a quantizing scheme.
The properties in the theorem might be 
rephrased as {\it the BT operator quantization has the
correct semiclassical limit}.

%%%%%%%%%%%%%%%%%%%%%%%%%%%%%%%%%%%%%%%%%%%%%%
\subsection{Further properties}
%%%%%%%%%%%%%%%%%%%%%%%%%%%%%
From  \refT{approx} (c) the 
\begin{proposition}
Let $f_1,f_2,\ldots,f_r\in C^\infty(M)$ then
\begin{equation}
||T^{(m)}_{f_1\ldots f_r}-
T^{(m)}_{f_1}\cdots
T^{(m)}_{f_r}||=O(m^{-1})
\end{equation}
\end{proposition}
\noindent
follows directly.
\begin{proposition}
\begin{equation}
\lim_{m\to\infty}
||\;[\Tfm,\Tgm]\;||\quad = \quad 0\ .
\end{equation}
\end{proposition}
\begin{proof}
Using the left side of the triangle inequality,
from  \refT{approx} (b) it follows that
$$
\left| m||\,[\Tfm,\Tgm]\,||-||\Tfgm||\right|\le 
||m\i[\Tfm,\Tgm]-\Tfgm||\quad=\quad O(\frac 1m)\ .
$$
By part (a) of the theorem 
$ ||\Tfgm||\to |\{f,g\}|_\infty$, 
and it stays finite. Hence 
$\ ||\,[\Tfm,\Tgm]\,||\ $ has to be be a zero sequence.
%\qed
\end{proof}
\begin{proposition}\label{P:sur}
The Toeplitz map 
$$
\Cim\to\eghm,\qquad  f\to\Tfm,
$$
is surjective.
\end{proposition}
\noindent
For a proof see \cite[Prop. 4.2]{BMS}.
\newline
This proposition says that for a fixed $m$ every operator 
$A\in\eghm$ is  the Toeplitz operator of a function
$f_m$. In the language of Berezin's co- and contravariant 
symbols $f_m$ will be the contravariant symbol of $A$.
We will discuss this in \refSS{symbols}. 
\begin{proposition}
For all $f\in\Cim$
$$
{T_f^{(m)}}^*=T_{\bar f}^{(m)}\ .$$
In particular, for real valued functions $f$ the associated 
Toeplitz operator is selfadjoint.
\end{proposition}
\begin{proof}
Take $s,t\in\ghm$ then 
$$
\scp {s}{\Tfm t}=
\scp {s}{\Pim (f\cdot t)}=
\scp {s}{f\cdot t}=
\scp {\bar f\cdot s}{t}=
\scp {T^{(m)}_{\bar f}s}{t}.
$$
\end{proof}
\noindent
The opposite of the last statement of the above proposition
is also true in the following sense.
\begin{proposition}\label{P:self}
Let $A\in\eghm$ be a selfadjoint operator then there exists a real valued
function $f$, such that $A=\Tfm$.
\end{proposition}
\begin{proof}
By the surjectivity of the Toeplitz map $A=\Tfm$ with a complex
valued function $f=f_0+\i f_1$ with real functions $f_0$ and $f_1$.
As $\Tfm=A=A^*=T^{(m)}_{\bar f}$ it follows $T_{f-\bar f}=0$ and hence
$T^{(m)}_{f_1}=0$.
From this we conclude $A=\Tfm=T^{(m)}_{f_1}$.
\end{proof}
We like to stress the fact that the Toeplitz map is never
injective on a fixed level $m$.
Only if $\ ||T^{(m)}_{f-g}|| \to 0\ $ for $m\to 0$ we can conclude that
$f=g$.
\begin{proposition}
Let $f\in\Cim$ and  $n=\dim_{\C}M$. Denote the 
trace on $\End(\ghm)$ by $\Tr^{(m)}$  then
\begin{equation}\label{E:trop}
\Tr^{(m)}\,(T^{(m)}_{f})
        =m^n\left(\frac 1{\volu (\P^{n}(\C))}
\int_M f\, \Omega +O(m^{-1})\right)\ .
\end{equation}
\end{proposition}
\noindent
See \cite{BMS}, resp. \cite{Schldef} for a detailed proof.

%%%%%%%%%%%%%%%%%%%%%%%%%%%%%%%%%%%%%%%%%%%%%%%%%%
\subsection{Strict quantization}
%%%%%%%%%%%%%%%%%%%%%%%%%%%%%
The asymptotic results of \refT{approx} says that 
the BT operator quantization is a strict quantization in the sense
of Rieffel 
\cite{Riefque}
as formulated in the book by Landsman \cite{Land}.
We take as base space $X=\{0\}\cup \{1/m\mid m\in\N\}$, 
with its induced topology coming from $\R$.
Note that $\{0\}$ is an accumulation point of the set
 $\{1/m\mid m\in\N\}$.
As $C^*$ algebras above the points $\{1/m\}$ we take the algebras
$\eghm$ and above  $\{0\}$ the algebra $\Cim$.
For $f\in\Cim$ we assign $0\mapsto f$, and $1/m\mapsto \Tfm$.
Now the property (a) in \refT{approx} is called in  \cite{Land}
Rieffel's condition, (b) Dirac's condition, and (c) von Neumann's
condition.
Completeness is true by Propositions \ref{P:sur} and \ref{P:self}.

 This definition  is closely related to the 
notion of continuous fields of $C^*$-algebras, see \cite{Land}.

%%%%%%%%%%%%%%%%%%%%%%%%%%%%%%%%%%%%%%%%%%%%%%%%%%%%%%%%%
\subsection{Relation to geometric quantization}
%%%%%%%%%%%%%%%%%%%%%%%%%%%%%%%%%
There exists another quantum operator in the geometric setting,
the operator of geometric quantization introduced by Kostant and Souriau.
In a first step the prequantum operator associated to the
bundle $L^m$ for the function $f\in\Cim$ is defined as
\begin{equation}
P_f^{(m)}:=\nabla_{X_f^{(m)}}^{(m)}+\i f\cdot  id.
\end{equation}
Here $\nabla^{(m)}$ is the connection in $L^m$, and  
$X_f^{(m)}$ the Hamiltonian vector field of $f$ with respect to the
K\"ahler form $\w^{(m)}=m\cdot \w$, i.e.
$m\w(X_f^{(m)},.)=df(.)$.
This operator $P_f^{(m)}$ acts on the space of differentiable global sections
of the line bundle $L^m$. 
The sections depend at every point on  $2n$ local coordinates and
one has to restrict the space to sections covariantly constant
along the excessive dimensions.
In technical terms,
one chooses a {\it polarization}. In general such a polarization is not 
unique. But in our complex situation there is 
canonical one 
by only taking the holomorphic sections.
This polarization is called {\it K\"ahler polarization}.
The operator of geometric quantization is then defined by
\begin{equation}\label{E:Geq}
Q_f^{(m)}:=\Pim P_f^{(m)}.
\end{equation}
The Toeplitz operator and the operator of geometric quantization
(with respect to the K\"ahler polarization)
are related by
\begin{proposition}\label{P:tuyn}
(Tuynman Lemma) Let $M$ be a compact quantizable K\"ahler 
manifold then 
\begin{equation}\label{E:tuyn}
Q_f^{(m)}=\i\cdot T_{f-\frac 1{2m}\Delta f}^{(m)},
\end{equation}
where $\Delta$ is the Laplacian with respect to the K\"ahler metric given by
$\omega$.
\end{proposition} 
\noindent
For the proof see  \cite{Tuyn}, and  \cite{BHSS}
for a coordinate independent proof.
\newline
In particular the $Q_f^{(m)}$ and the $\Tfm$ have the same asymptotic 
behaviour. We obtain for $Q_f^{(m)}$ similar results as in 
\refT{approx}. For details see \cite{Schlhab}.
It should be noted that for \refE{tuyn} the compactness of $M$ is essential.

%%%%%%%%%%%%%%%%%%%%%%%%%%%%%%%%%%%%%%%%%%%%%%%%%%%%%%%%%
\subsection{$L_\alpha$ approximation}
%%%%%%%%%%%%%%%%%%%%%%%%%%%%%%%%%%%%%%%%%%%%
In \cite{BHSS}  the notion of $L_{\alpha}$, resp.
$gl(N)$, resp. $su(N)$ quasi-limit were introduced.
It was conjectured in \cite{BHSS} that for every 
compact quantizable K\"ahler manifold the Poisson algebra of functions 
is a $gl(N)$ quasi-limit. 
In fact, the conjecture follows from the \refT{approx}, see
\cite{BMS} and \cite{Schlhab} for details.
%%%%%%%%%%%%%%%%%%%%%%%%%%%%%%%%%%%%%%%%%%%%%%%%%%%%%
\subsection{The noncompact situation}
\label{SS:noncompact}
%%%%%%%%%%%%%%%%%%%%%%%%%%%%%%%%%%%%%%%%%%%%%%%%%%%%
Berezin-Toeplitz operators can be introduced  for non-compact
K\"ahler manifolds.
In this case the $L^2$ spaces are the space of bounded sections
and for the subspaces of holomorphic sections one can only consider
the
bounded holomorphic sections. Unfortunately,
in this context the proofs of \refT{approx} 
do not work. One has to study examples or classes of examples 
case by case  whether the corresponding
properties are correct.

In the following we give a very incomplete list of
references.
Berezin himself studied bounded complex-symmetric domains
 \cite{Beress}.
In this case the manifold is an open domain in 
$\C^n$.
Instead of sections one studies functions which are integrable
with respect to a suitable measure depending on
$\hbar$. Then  $1/\hbar$ 
corresponds to the tensor power of our bundle.
Such Toeplitz operators were studied extensively by  Upmeier in 
a series of works 
\cite{Upa,Upb,Upc,Upd}.
See also the book of Upmeier
\cite{UpB}.
For $\C^n$ see Berger and Coburn \cite{BeCob}, \cite{Cob}.
Klimek and Lesniewski \cite{KlLeqr}  studied  
the Berezin-Toeplitz quantization
on the unit disc. Using automorphic forms and the universal
covering they obtain results  for Riemann surfaces of genus 
$g\ge 2$.
The names of  Borthwick, Klimek, Lesniewski, Rinaldi, and  Upmeier
should be mentioned in the context of BT quantization for 
Cartan domains and super Hermitian spaces.

A quite different approach to Berezin-Toeplitz quantization is based
on  the asymptotic expansion of the Bergman kernel outside 
the diagonal. This was also used  by the author together 
with Karabegov \cite{KS} for the compact K\"ahler case.
See \refS{btrans} for some details.
Engli\v s \cite{Engbk} showed similar results for bounded 
pseudo-convex domains in $\C^N$.
Ma and Marinescu \cite{MaMar,MaMar1} developed a 
theory of Bergman kernels
for the symplectic case, which yields also results on the 
Berezin-Toeplitz operators  for certain  non-compact K\"ahler manifolds
and even orbifolds.

%%%%%%%%%%%%%%%%%%%%%%%%%%%%%%%%%%
\section{Berezin-Toeplitz deformation quantization}
\label{S:btstar}
%\input btstar.tex
%%%   Berezin-Toeplitz deformation quantization
%%%%%%%%%%%%%%%%%%
%%%%  22.12.2009/29.12./30.12
%%%%%%%%%%%%%%%%%%%%%%%%%%%%%%%%%%%%%%%%%%%%%%%%%%%
There is another approach to quantization. Instead of assigning
noncommutative operators to commuting functions one
might think about ``deforming'' the pointwise commutative 
multiplication of functions into a non-commutative product. 
It is required to remain associative, the commutator of
two elements 
should relate to the Poisson bracket of the elements, and it should
reduce in the ``classical limit'' to the commutative 
situation.

It turns out that such a deformation which is valid for all 
differentiable functions cannot exist. A way out is to
enlarge the algebra of functions by considering formal 
power series over them and to deform the product inside this
bigger algebra.
A first systematic treatment and  applications in physics of 
this idea were given 1978 
by
Bayen, Flato, Fronsdal, Lichnerowicz, and Sternheimer
 \cite{BFFLS}. There the notion
of {\it deformation quantization} and {\it star products} were
introduced.
Earlier versions of these concepts were around
due to  Berezin \cite{Berequ}, Moyal \cite{Moy}, and 
Weyl \cite{Weyl}.
For a presentation of the history see \cite{Stern}.

We will show that for compact K\"ahler manifolds $M$, 
there is a natural star product.

%%%%%%%%%%%%%%%%%%%%%%%%%%%%%%
\subsection{Definition of star products}
%%%%%%%%%%%%%%%%%%%%%%%%%%%%%
We start with a Poisson manifold $(M,\{.,.\})$, i.e. a differentiable
manifold with a Poisson bracket for the function such that
$(\Cim,\cdot, \{.,.\})$ is a Poisson algebra.
Let $\mathcal{A}=\Cim[[\nu]]$ be the algebra of formal power
 series in the 
variable $\nu$ over the algebra $\Cim$.
\begin{definition}\label{D:star}
A product $\star$
 on $\mathcal {A}$ is 
called a (formal) star product for $M$ (or for $\Cim$) if it is an
associative $\C[[\nu]]$-linear product which is $\nu$-adically continuous 
such that
\begin{enumerate}
\item
\qquad $\mathcal{ A}/\nu\mathcal {A}\cong\Cim$, i.e.\quad $f\star g
 \bmod \nu=f\cdot g$,
\item
\qquad $\dfrac 1\nu(f\star g-g\star f)\bmod \nu = -\i \{f,g\}$,
\end{enumerate}
where $f,g\in\Cim$.
\end{definition}
Alternatively we can  
write 
\begin{equation}
\label{E:cif}
 f\star g=\sum\limits_{j=0}^\infty C_j(f,g)\nu^j\ ,
\end{equation}
with
$ C_j(f,g)\in\Cim$ such that the  $C_j$ are bilinear in the entries $f$ and $g$.
The conditions (1) and (2)  can 
be reformulated as 
\begin{equation}
\label{E:cifa}
C_0(f,g)=f\cdot g,\qquad\text{and}\qquad
C_1(f,g)-C_1(g,f)=-\i \{f,g\}\ .
\end{equation}
By the $\nu$-adic continuity \refE{cif} fixes $\star$ on  $\mathcal{A}$.
A {\em (formal) deformation quantization} is given by a {\it (formal) 
star product}.
I will use both terms interchangeable.

There are certain additional conditions 
for a star product which are sometimes useful.
\begin{enumerate}
\item We call it  ``null on constants'', if  $1\star f=f\star 1=f$,
which is equivalent to the fact
that the constant function $1$ will remain the unit
in $\mathcal{A}$.
In terms of the coefficients it can be formulated as
$C_k(f,1)=C_k(1,f)=0$ for $k\ge 1$.
In this review we always assume this to 
be the case for  star products.
\item
We call it selfadjoint if 
$\ \overline{f\star g}=\overline{g}\star\overline{f}$,
where we assume $\bar \nu=\nu$.
\item
We call it local if 
$$\mathrm{supp}\, C_j(f,g)\subseteq \mathrm{supp}\, f\cap
\mathrm{supp}\, g,
\qquad 
\forall f,g\in\Cim.
$$
{}From the locality property it follows that 
the $C_j$ are bidifferential operators and that 
the global star product defines
for every open subset $U$ of $M$ a star product for 
the Poisson algebra $C^\infty(U)$.
Such local star products are also called
{\it differential star products}.
\end{enumerate}

%%%%%%%%%%%%%%%%%%%%%%%%%
\subsection{Existence of star products}
%%%%%%%%%%%%%%%%%%%%%%%%%%%%%
In the usual setting of deformation theory there always exists 
a trivial deformation. This is not the case here, as the trivial
deformation of $\Cim$ to  $\mathcal{A}$, which is nothing else as extending
the point-wise product to the power series, is not
allowed as it does not fulfil Condition (2) in \refD{star}
(at least not if the Poisson bracket is non-trivial).
In fact the existence problem is highly non-trivial.
In the symplectic case different existence proofs, from different perspectives,
were given by DeWilde-Lecomte \cite{DeWiLe},
Omori-Maeda-Yoshioka \cite{OMY},
and 
Fedosov \cite{Fed}.
The general Poisson case was settled by Kontsevich
\cite{Kont}.

%%%%%%%%%%%%%%%%%%%%%%%%%%%%%%%%%%%%%%%%%%%%%%%%%%%%%%%
\subsection{Equivalence and classification of star products}
%%%%%%%%%%%%%%%%%%%%%%%%%%%%%%%%%%%%%%%%%%%%%%%%%%
\begin{definition}
Given a Poisson manifold $(M,\{.,.\})$. Two star products
$\star$ and $\star'$ associated to the Poisson structure $\{.,.\}$ 
are called equivalent if and only if  there exists   a formal series of 
linear operators
\begin{equation}
B=\sum_{i=0}^\infty B_i\nu^i,\qquad
B_i:\Cim\to\Cim,
\end{equation}
with $\ B_0=id\ $ such that
\begin{equation}
B(f)\star' B(g)=B(f\star g).
\end{equation}
\end{definition}
For local star products 
in the general Poisson setting
there are complete classification results.
Here I will only consider the symplectic case.

To each local star product $\star$ its {\it Fedosov-Deligne class} 
\begin{equation}
cl(\star)\in \frac {1}{\i \nu}[\w]+ H^2_{dR}(M)[[\nu]]
\end{equation}
can be assigned. Here  $H^2_{dR}(M)$ denotes the 2nd deRham
cohomology class of closed 2-forms modulo exact forms
and  $H^2_{dR}(M)[[\nu]]$ the formal power series with such classes
as coefficients.
Such formal power series are called 
{\it formal deRham classes}.
In general we will use $[\a]$ for the cohomology 
class of a form $\a$.

This assignment gives a 1:1 correspondence between the 
formal deRham classes and the equivalence classes of star 
products.

For contractible manifolds  we have $ H^2_{dR}(M)=0$ and hence
there is up to equivalence exactly one local star product.
This yields that locally all local star products of a manifold are
equivalent to a certain fixed one, which is called the Moyal product.
For these and related classification results see \cite{Delstar},
\cite{GuRa}, \cite{BeCaGu}, \cite{NeTs}.

%%%%%%%%%%%%%%%%%%%%%%%%%%%%%%%%%%%%%%%%%%%%%%%%%%
\subsection{Star products with separation of variables}
%%%%%%%%%%%%%%%%%%%%%%%%%%%%%%%%%%%%%%%
For our compact K\"ahler manifolds we will have many 
different and even non-equivalent star products. The question is:
is there a star product which is given in a natural way?
The answer will be yes: the Berezin-Toeplitz star product to be
introduced below.
First we consider star products respecting the complex structure in
a certain sense.
\begin{definition}\label{D:sep} (Karabegov \cite{Kar1})
A star product is called {\it star product with separation
of variables} if and only if 
\begin{equation}
f\star h=f\cdot h,\quad\text{and}\quad
h\star g=h\cdot g,
\end{equation}
for every locally defined holomorphic function $g$, antiholomorphic
function $f$, and arbitrary function $h$.
\end{definition}
Recall that  a local star product $\star$ for $M$ defines a star product for
every open subset $U$ of $M$. We have just to take the bidifferential
operators defining   $\star$.
Hence it makes sense to talk about $\star$-multiplying with
local functions.
\begin{proposition}
A local $\star$ product has the separation of variables 
property if and only if in the bidifferential operators
$C_k(.,.)$ for $k\ge 1$ 
in the first argument only derivatives in holomorphic and in the
second argument only derivatives in antiholomorphic directions
appear.
\end{proposition}
In Karabegov's original notation the r\^oles of the holomorphic
and antiholomorphic functions is switched.
Bordemann and Waldmann \cite{BW} called such star products
{\it star products of Wick type}. Both
Karabegov and Bordemann-Waldmann proved that there exists 
for every K\"ahler manifold star products of separation of
variables type.
In \refSS{kara} we will give more details on Karabegov's
construction.
Bordemann and Waldmann modified Fedosov's method \cite{Fed} to
obtain such a star product.
See also Reshetikhin and Takhtajan \cite{ResTak}
for yet another construction.
But I like to point out that in  all these constructions the
result is only a formal star product without any relation 
to an operator calculus, which will be given by 
the Berezin-Toeplitz star product introduced in the
next section.

Another  warning is in order.
The property of being a star product of separation of variables
type will not be kept by equivalence transformations.

%%%%%%%%%%%%%%%%%%%%%%%%%%%%%%%%%%%%%%%%%%%%%%%%%%%%%%%%%%%%%%%%%%
\subsection{Berezin-Toeplitz star product}
%%%%%%%%%%%%%%%%%%%%%%%%%%%%%%%%%%%%%%%%%%%%%%%%%%%
%%%%%%%%%%%%%%%%%%%%%%%%%%%%%%%%%%%%%%%%%%%%%%%% 
\begin{theorem}
\label{T:star}
There exists a unique (formal) star product $\star_{BT}$ for $M$
\begin{equation}
f \star_{BT} g:=\sum_{j=0}^\infty \nu^j C_j(f,g),\quad C_j(f,g)\in
C^\infty(M),
\end{equation}
in such a way that for  $f,g\in\Cim$ and for every $N\in\N$  we have
with suitable constants $K_N(f,g)$ for all $m$
\begin{equation}
\label{E:sass}
||T_{f}^{(m)}T_{g}^{(m)}-\sum_{0\le j<N}\left(\frac 1m\right)^j
T_{C_j(f,g)}^{(m)}||\le K_N(f,g) \left(\frac 1m\right)^N\ .
\end{equation}
The star product is null on constants and selfadjoint.
\end{theorem}
\noindent
This theorem has been proven immediately after 
\cite{BMS} was finished. It has been announced in \cite{Schlbia95},%
\cite{SchlGos}
and the proof was written up in German in  \cite{Schlhab}.
A complete proof published in English 
can be found in \cite{Schldef}.

For simplicity we might write
\begin{equation}
\label{E:expans}
T_{f}^{(m)}\cdot T_{g}^{(m)}
\quad \sim\quad \sum_{j=0}^\infty\left(\frac 1m\right)^j
T_{C_j(f,g)}^{(m)}
\qquad (m\to\infty),
\end{equation}
but 
we will always assume the strong and precise statement of \refE{sass}.
The same is assumed for other asymptotic formulas appearing 
further down in this review.

Next we want to identify this star product. Let $K_M$ be the 
canonical line bundle of $M$, i.e. the $n^{th}$ exterior power of
the holomorphic 1-differentials. The canonical class $\delta$ is
the first Chern class of this line bundle, i.e. 
$\delta:= c_1(K_M)$.
If we take in $K_M$ the fibre metric coming from the 
Liouville form $\Omega$ then this defines a unique
connection and further a unique curvature $(1,1)$-form
$\w_{can}$. 
In our
sign conventions we have  $\delta=[\w_{can}]$.

Together with Karabegov the author showed
\begin{theorem}\label{T:bstarad} \cite{KS}
(a) 
The Berezin-Toeplitz star product is a local star product which 
is of separation of variable type.
\newline
(b) 
Its classifying Deligne-Fedosov class is
\begin{equation}\label{E:starcl}
cl(\star_{BT})=\frac {1}{\i}\left(\frac 1\nu[\w]-\frac {\delta}{2}\right)
\end{equation} 
for the characteristic class of the star product $\star_{BT}$.
\newline
(c) 
The classifying Karabegov form
associated to the  
Berezin-Toeplitz star product is
\begin{equation}\label{E:starform}
-\frac 1\nu\w+\w_{can}.
\end{equation}
\end{theorem}
The Karabegov form has not yet defined here.
We will introduce it below in \refSS{kara}.
Using $K$-theoretic methods the formula for $cl(\star_{BT})$ was also given by
Hawkins \cite{Haw}.

%%%%%%%%%%%%%%%%%%%%%%%%%%%%%%%%%%%%%%%%%%%
\subsection{Star product of geometric quantization}
%%%%%%%%%%%%%%%%%%%%%%%%%%%%%%%%%%%
Tuynman's result \refE{tuyn} relates
the operators of geometric quantization with
K\"ahler polarization and  
the  BT operators. As the latter define a star product
it can be used to give also a star product $\star_{GQ}$ 
associated to geometric quantization.
Details can be found in \cite{Schldef}.
This star product will be equivalent to  the BT star product,
but it is not of separation of variables type.
The equivalence is given by the  
$\C[[\nu]]$-linear map induced by
\begin{equation}
B(f):=f-\nu\frac {\Delta}{2} f=
(id-\nu\frac {\Delta}{2})f. 
\end{equation}
We obtain 
$B(f)\star_{BT} B(g)=B(f\star_{GQ} g)$.

%%%%%%%%%%%%%%%%%%%%%%%%%%%%%%%%%%%%%%%%%%%%%%%%%
\subsection{Trace for the BT star product}
%%%%%%%%%%%%%%%%%%%%%%%%%%%%%%%%%%%%%%%%
{}From \refE{trop}  the following complete 
asymptotic expansion for $m\to\infty$ 
can be deduced \cite{Schldef}, \cite{BPUspec}):
\begin{equation}
\label{E:tras}
\Tr^{(m)}(T_f^{(m)})\quad\sim\quad
m^n\left(\sum_{j=0}^\infty \left(\frac 1{m}\right)^j\tau_j(f)\right),
\quad \mathrm{with}\quad \tau_j(f)\in\C\ .
\end{equation}
We define the $\C[[\nu]]$-linear map
\begin{equation}
\label{E:trace}
\Tr:\Cim[[\nu]]\to \nu^{-n}\,\C[[\nu]],\quad
\Tr f:=\nu^{-n}\sum_{j=0}^\infty\nu^j\tau_j(f),
\end{equation}
where the $\tau_j(f)$ are given
by the asymptotic expansion \refE{tras} for $f\in\Cim$
and for arbitrary elements by $\C[[\nu]]$-linear extension.
\begin{proposition} \cite{Schldef}
The map $\Tr$ is a trace, i.e., we have
\begin{equation}
\label{E:trsym}
\Tr (f\star g)=\Tr (g\star f)\ .
\end{equation}
\end{proposition}

%%%%%%%%%%%%%%%%%%%%%%%%%%%%%%%%%%%%%%%%%%%%%%%%%
\subsection{Karabegov quantization}\label{SS:kara}
%%%%%%%%%%%%%%%%%%%%%%%%%%%%%%%%%%%%%%%%%%%%%%%%%%%%%%%%
In \cite{Kar1,Kar2} Karabegov not only gave the notion of
{\it separation of variables type}, but also a proof of 
existence of such formal star products for any 
K\"ahler manifold, whether compact, non-compact, quantizable, or
non-quantizable.
Moreover, he classified them completely as individual 
star product not only up to equivalence.

He starts with $(M,\omega_{-1})$ a pseudo-K\"ahler manifold, i.e.
a complex manifold with a non-degenerate closed $(1,1)$-form not
necessarily positive.

A formal form $\widehat{\omega}=
(1/\nu)\omega_{-1}+\omega_0+\nu\omega_1+\dots$ is
called a formal deformation of the form $(1/\nu)\omega_{-1}$ if the forms
$\omega_r,\ r\geq 0$, are closed but not necessarily nondegenerate
(1,1)-forms on $M$.
It was shown in \cite{Kar1} that all deformation quantizations with separation
of variables on the  pseudo-K\"ahler manifold $(M,\omega_{-1})$ are
bijectively parametrized by the formal deformations of the form
$(1/\nu)\omega_{-1}$. 

Assume that we have such a star product 
$(\mathcal{A}:=\cim[[\nu]],\star)$. Then 
for $f,g\in\mathcal{A}$ the  operators of left
and right multiplication  $L_f,R_g$  are given by  
$L_fg=f\star g=R_gf$. The
associativity of the star-product $\star$ is equivalent to the
fact that $L_f$ commutes with $R_g$ for all $f,g\in{\mathcal{A}}$.
If a star  product is differential then 
$L_f,R_g$ are formal differential operators.

Karabegov constructs his star product 
associated to the deformation $\widehat{\w}$ in the following way.
First he chooses on every 
contractible
coordinate chart $U\subset M$ (with holomorphic
coordinates $\{z_k\}$)
its formal potential  
\begin{equation}\label{E:formp}
\widehat{\Phi}=(1/\nu)\Phi_{-1}+\Phi_0+\nu\Phi_1+\dots,
\qquad
\widehat{\omega}=i\partial\bar\partial\widehat{\Phi}.
\end{equation}
Then construction is done in such a way that we have for 
the left (right) multiplication operators   on $U$
\begin{equation}
L_{\partial\Phi/\partial
z_k}=\partial\Phi/\partial
z_k+\partial/\partial z_k, 
\quad\text{and}\quad 
R_{\partial\Phi/\partial\bar
z_l}=\partial\Phi/\partial\bar
z_l+\partial/\partial\bar z_l.
\end{equation} 
The set
$\mathcal{L}(U)$ of all left multiplication
operators on $U$ is completely described as the
set of all formal differential operators
commuting with the point-wise multiplication
operators by antiholomorphic coordinates
$R_{\bar z_l}=\bar z_l$ and the operators
$R_{\partial\Phi/\partial\bar
z_l}$. From the knowledge of $\mathcal{L}(U)$ the 
star product on $U$ can be reconstructed.
The local
star-products agree on the intersections of the
charts and define the global star-product
$\star$ on $M$.

We have to mention that  this  original
 construction of Karabegov will yield a star product of separation of
variable type but with the role of holomorphic and antiholomorphic
variables switched.
This says 
  for any open
subset $U\subset M$ and any holomorphic function $a$ and
antiholomorphic function $b$ on $U$ the operators $L_a$ and
$R_b$ are the operators of point-wise multiplication by $a$ and
$b$ respectively, i.e., $L_a=a$ and $R_b=b$. 

%%%%%%%%%%%%%%%%%%%%%%%%%%
\subsection{Karabegov's formal Berezin transform}
%%%%%%%%%%%%%%%%%%%%%%%%%%%%%%%%%%%%%%

Given such a  star products $\star$, Karabegov introduced the formal
{\it Berezin transform} $I$
as the unique formal differential operator on
$M$ such that for any open subset $U\subset M$,
holomorphic functions $a$ and antiholomorphic
functions $b$ on $U$ the relation $I(a\cdot b)=b\star
a$ holds (see \cite{Kar3}). 
He shows  that
$I=1+\nu\Delta+\dots$, where $\Delta$ is the
Laplace operator corresponding to the
pseudo-K\"ahler metric on $M$. 

Karabegov  considered the following associated star products 
First the {\it dual}
star-product $\tilde\star$ on $M$ is defined for
$f,g\in \mathcal{A}$ by the formula 
\begin{equation}
f\,\tilde\star\, g=I^{-1}(Ig\star If).
\end{equation}
It is  a star-product with separation of
variables on the pseudo-K\"ahler manifold
$(M,-\omega_{-1})$. Its  formal Berezin transform equals
$I^{-1}$, and thus the dual to $\ \tilde\star\ $ is
$\ \star\ $.
Note that it is not a star product of the same pseudo-K\"ahler 
manifold.
Denote by $\tilde\omega=-(1/\nu)\omega_{-1}
+\tilde\omega_0+\nu\tilde\omega_1+\dots$ the
formal form parametrizing the star-product
$\tilde\star$.  

Next, the opposite of the dual
star-product, $\star'=\tilde\star^{op}$, is given
by the formula 
\begin{equation}
f\star' g=I^{-1}(If\star Ig).
\end{equation}
It  defines a deformation quantization with
separation of variables on $M$, but with the
roles of holomorphic and antiholomorphic
variables swapped - with respect to $\star$. 
It could be described also as 
 a deformation quantization with
separation of variables on the pseudo-K\"ahler
manifold $(\overline{M},\omega_{-1})$ where
$\overline{M}$ is the manifold $M$ with the
opposite complex structure.
But now the pseudo-K\"ahler form will be the same.
Indeed  
the formal Berezin
transform $I$ establishes an equivalence of
deformation quantizations $(\mathcal{A},\star)$ and
$(\mathcal{A},\star')$. 

How is the relation to the Berezin-Toeplitz star product 
$\star_{BT}$ of \refT{star}?
There exists a certain formal deformation 
$\widehat{\w}$ of the form
$(1/\nu) \w$ which yields a star product $\star$ in the Karabegov sense.
The opposite of its dual will be equal to the 
Berezin-Toeplitz star product, i.e. 
\begin{equation}
\star_{BT}\ =\ \tilde\star^{op}\ =\ \star '\ .
\end{equation}
The classifying Karabegov form $\ \tilde\w\ $ of $\ \tilde\star\ $  
will be the form \refE{starform}.
Note as $\ \star\ $ and $\ \star_{BT}\ $ are equivalent via $\ I\ $, we have 
 $cl(\star)=cl(\star_{BT})$, see the formula \refE{starcl}.
We will identify $\widehat{\w}$ in \refSS{ident}.

%%%%%%%%%%%%%%%%%%%%%%%%%%%%%%%%%%%%%%%
\section{The disc bundle and global operators}
\label{S:global}
%\input global.tex
%%%%%%%%%%%%%%%%%%%%%%%%%%%%%%%%%%%%%%%%%
%%%%  The disc bundle and global operators
%%%%%%%%%%%%%%%%%%%%%%%%%%%%%%
%%%%%%%%%%%%%%    25.12.09/ 29.12.2009/30.12.
%%%%%%%%%%%%%%%%%%%%%%%%%%%%%%%%%%%%%%%%%%
In this section we identify the bundles $L^m$ over the K\"ahler
manifold $M$ as associated line bundles of one unique 
$S^1$-bundle over $M$.
The Toeplitz operator will appear as ``modes'' of a global Toeplitz
operator. A detailed analysis of this global operator will yield
a proof of \refT{approx} part (b) and part (c).

Moreover, we will need this set-up to discuss coherent states, 
Berezin symbols, and the Berezin transform in the next sections.
For a more detailed presentation see
\cite{Schlhab}.

%%%%%%%%%%%%%%%%%%%%%%%%%%%%%%%%%%%%%%%%%%%%%%%%
\subsection{The disc bundle}
%%%%%%%%%%%%%%%%%%%%%%%%%%%%%%%%%%%%%%%%
We will assume that the quantum line bundle $L$ is already very ample, i.e. 
it has enough global holomorphic sections 
to embed $M$ into projective space.
From the bundle% 
\footnote{As the connection $\nabla$ will not be needed anymore, I 
will drop it in the notation.}
$(L,h)$  we pass to its dual
$\ (U,k):=(L^*,h^{-1})\ $ 
with dual metric $k$.
Inside of the total space $U$ we consider the circle bundle
\begin{equation}
Q:=\{\lambda\in U\mid k(\lambda,\lambda)=1\},
\end{equation} 
the (open) disc bundle and (closed) disc bundle respectively
\begin{equation}
D:=\{\lambda\in U\mid k(\lambda,\lambda)<1\},\qquad
\overline{D}:=\{\lambda\in U\mid k(\lambda,\lambda)\le 1\}.
\end{equation} 
Let 
$\tau: U\to M$ the projection (maybe restricted to the subbundles).

For the projective space $\P^N(\C)$ with 
the hyperplane section bundle  $H$
as quantum line bundle 
the bundle $U$ is just the tautological
bundle. Its  fibre over the point $z\in\P^N(\C)$ consists of
the line in $\C^{N+1}$ which is represented by $z$. In particular,
for the projective space 
the total space of $U$ with  the zero section removed can be identified
with $\C^{N+1}\setminus\{0\}$.
The same picture remains true for the via the very ample quantum
line bundle in projective space embedded manifold $M$.
The quantum line bundle will be the pull-back of $H$ 
(i.e. its restriction to the embedded manifold) and its
dual is the pull-back of the tautological bundle.

In the following we use $E\setminus 0$ to denote the total space of
a vector bundle $E$ with the image of the zero section removed. 
Starting from the  real valued function 
$\hat k(\la):=k(\la,\la)$ on $U$ we define
 $\tilde a:=\frac {1}{2\i}(\d-\pbar)\log \hat k$ on
$U\setminus 0$ (the derivation are taken with respect to the complex
structure on $U$) and denote by $\alpha$
its restriction   to $Q$.
With the help of the quantization condition \refE{quant} we obtain 
$d\a=\tau^*\w$ (with the deRham differential $d=d_Q$) 
and that in fact $\mu=\frac 1{2\pi}\tau^*\Omega
\wedge \a$ is a volume form
on $Q$. 
Indeed $\a$ is a contact form for the contact manifold $Q$.
As far as the integration is concern we get 
\begin{equation}
\int_Q(\tau^*f)\mu=\int_Mf\,\Omega,\qquad \forall f\in\Cim.
\end{equation}
Recall that $\Omega$ is the Liouville volume form on $M$.

%%%%%%%%%%%%%%%%%%%%%%%%%%%%%%%%%%%%%%%%%%%%%%%%
\subsection{The generalized Hardy space}
%%%%%%%%%%%%%%%%%%%%%%%%%%%%%%%%%
With respect to $\mu$ we take the L${}^2$-completion $\Lqv$
of the space of functions on $Q$.
The generalized {\em Hardy space} $\Hc$ is the closure of the
space of those 
functions in  $\Lqv$ which can be extended to
holomorphic functions on the whole
disc bundle
$\bar D$.
The generalized {\em Szeg\"o projector} is the projection
\begin{equation}
\label{E:szproj}
\Pi:\Lqv\to \Hc\ .
\end{equation}
 By the natural circle action the bundle 
$Q$ is an $S^1$-bundle and the tensor powers of $U$ can be
viewed as associated line bundles. The space $\Hc$ is preserved
by the $S^1$-action.
It can be decomposed into eigenspaces 
$\Hc=\prod_{m=0}^\infty \Hm$ where
 $c\in S^1$ acts on $\Hm$ as multiplication
by $c^m$.
The Szeg\"o projector is $S^1$ invariant and 
can be decomposed into its components, the Bergman projectors  
\begin{equation}\label{E:berg}
\pimh:\Lqv\to\Hm.
\end{equation}

Sections of $L^m=U^{-m}$ can be identified with functions $\psi$ on $Q$ which
satisfy the equivariance condition
$\psi(c\la)=c^m\psi(\la)$, i.e. which are homogeneous of degree $m$.
This identification is given via the map
\begin{equation}\label{E:secident}
\gamma_m:\Lpm \to \Lqv,\quad s\mapsto \psi_s\quad\text{where}\quad
\psi_s(\alpha)=\alpha^{\otimes m}(s(\tau(\alpha))),
\end{equation}
which turns out to be an isometry onto its image.
On $\Lpm$ we have the  scalar product  \refE{skp}.
Restricted to the holomorphic sections we obtain the 
isometry 
\begin{equation}\label{E:sechident}
\gamma_m:\ghm \cong \Hm.
\end{equation} 
In the case of $\P^N(\C)$ this correspondence is nothing else
as the identification of the global sections of the $m^{th}$ tensor
powers of the hyper plane section bundle with the homogenous 
polynomial functions of degree $m$ on $\C^{N+1}$.

%%%%%%%%%%%%%%%%%%%%%%%%%%%%%%%%%%%%%%%%%%%%%%%%%%%%%%%%%%%%%%
\subsection{The Toeplitz structure}
%%%%%%%%%%%%%%%%%%%%%%%%%%%%%%%%%%%%%%%%%%%%%%%%
There is the notion of Toeplitz structure
$(\Pi,\Sigma)$ as developed by Boutet de Monvel  and
Guillemin
in \cite{BGTo,GuCT}.
I do not want to present the general theory only the
specialization to our situation.
Here 
$\Pi $ is the  Szeg\"o projector  \refE{szproj}
and $\Sigma$  is  the submanifold 
\begin{equation}
\Sigma:=\{\;t\alpha(\lambda)\;|\;\lambda\in Q,\,t>0\ \}\ \subset\  T^*Q
\setminus 0
\end{equation}
of the tangent bundle of $Q$ 
defined with the help of the 1-form $\alpha$.
It turns out that  $\Sigma$ is a 
symplectic submanifold, called a symplectic cone.

A (generalized) {\em Toeplitz operator} of order $k$ is  an operator
$A:\Hc\to\Hc$ of the form
$\ A=\Pi\cdot R\cdot \Pi\ $ where $R$ is a
pseudodifferential operator ($\Psi$DO)
of order $k$ on
$Q$.
The Toeplitz operators constitute a ring.
The  symbol of $A$ is the restriction of the
principal symbol of $R$ (which lives on $T^*Q$) to $\Sigma$.
Note that $R$ is not fixed by $A$, but
Guillemin and Boutet de Monvel showed that the  symbols
are well-defined and that they obey the same rules as the
symbols of   $\Psi$DOs.
In particular, the following relations are valid:
\begin{equation}
\label{E:symbol}
\sigma(A_1A_2)=\sigma(A_1)\sigma(A_2),\qquad
\sigma([A_1,A_2])=\i\{\sigma(A_1),\sigma(A_2)\}_\Sigma.
\end{equation}
Here $\{.,.\}_{\Sigma}$ is the restriction of the canonical
Poisson structure of $T^*Q$ to $\Sigma$ coming from the
canonical symplectic form  on $T^*Q$.

%%%%%%%%%%%%%%%%%%%%%%%%%%%%%%%%%%%%%%%%%%%%%%%%%%%%%%%%%%%%%%%%%%%%%%
\subsection{A sketch of the proof of \refT{approx}}
%%%%%%%%%%%%%%%%%%%%%%%%%%%%%%%%%%%%%%%%%%%%
For this we need only to consider the following 
two generalized Toeplitz operators:
\begin{enumerate}
\item
The generator of the circle action
gives the  operator $D_\varphi=\dfrac 1{\i}\dfrac {\partial}
{\partial\varphi}$, where $\varphi$ is the angular variable. 
It is an operator of order 1 with symbol $t$.
It operates on $\Hm$ as multiplication by $m$.
\item
For $f\in\Cim$ let $M_f$ be the  operator on
$\Lqv$ 
corresponding to multiplication with $\tau^*f$.
We set
\begin{equation}
T_f=\Pi\cdot M_f\cdot\Pi:\quad \Hc\to\Hc\ .
\end{equation}
As $M_f$ is constant along the fibres of $\tau$, 
the operator $T_f$ 
commutes with the circle action.
Hence we can decompose
\begin{equation}
T_f=\prod\limits_{m=0}^\infty\Tfm\ ,
\end{equation}
where $\Tfm$ denotes the restriction of $T_f$ to $\Hm$.
After the identification of $\Hm$ with $\ghm$ we see that these $\Tfm$
are exactly the Toeplitz operators  $\Tfm$ introduced in \refS{btq}.
We call   $T_f$   the global Toeplitz operator and
the $\Tfm$ the local Toeplitz operators.
The operator $T_f$ is  of order $0$.
Let us denote by
$\ \tau_\Sigma:\Sigma\subseteq T^*Q\to Q\to M$ the composition
then we obtain for its symbol  $\sigma(T_f)=\tau^*_\Sigma(f)$.
\end{enumerate}
%%%%%%%%%%%%%%%%%%%%%%%%%%%%%%%%%%%%%%%%%%%%%%%%%%%%%% 
Now we are able to proof \refE{dirac}.
First we introduce for a fixed $t>0$ 
\begin{equation}
\Sigma_t:=\{t\cdot \alpha(\la)\mid \la\in Q\}\quad \subseteq \Sigma.
\end{equation}
It turns out that 
 $\ {\omega_\Sigma}_{|\Sigma_t}=-t\tau_\Sigma^*\omega\ $.
The commutator
$[T_f,T_g]$ is a  Toeplitz operator of order $-1$.
From the above  we obtain
with \refE{symbol} that the 
symbol of the commutator equals
\begin{equation}
\sigma([T_f,T_g])(t\alpha(\lambda))=\i\{\tau_\Sigma^* f,\tau_\Sigma^*g
\}_\Sigma(t\alpha(\lambda))=
-\i t^{-1}\{f,g\}_M(\tau(\lambda))\ .
\end{equation}
We consider the Toeplitz operator
\begin{equation}
A:=D_\varphi^2\,[T_f,T_g]+\i D_\varphi\, T_{\{f,g\}}\ .
\end{equation}
Formally this is an operator of order 1.
Using $\ \sigma(T_{\{f,g\}})=\tau^*_\Sigma \{f,g\}$ 
and $\sigma(D_\varphi)=t$ we see that its principal
symbol vanishes. Hence
it is an operator of order 0.
Now $M$ and hence also $Q$ are compact manifolds.
 This implies that $A$ is a bounded
operator ($\Psi$DOs of order 0 on compact manifolds are bounded).
It is obviously $S^1$-invariant and we can write
$A=\prod_{m=0}^\infty A^{(m)}$
where $A^{(m)}$ is the restriction of $A$ on the space $\Hm$.
For the norms we get $\ ||A^{(m)}||\le ||A||$.
But
\begin{equation}
A^{(m)}=A_{|\Hm}=m^2[\Tfm,\Tgm]+\i m\Tfgm.
\end{equation}
Taking the norm bound and dividing it by $m$ we get 
part (b) of \refT{approx}.
Using \refE{sechident} the norms involved indeed coincide.

Quite similar one can prove part (c) of  \refT{approx} and
more general the existence of the coefficients $C_j(f,g)$ 
for the Berezin-Toeplitz star product of \refT{star}.
See \cite{Schldef} and \cite{Schlhab} for the details.

%%%%%%%%%%%%%%%%%%%%%%%%%%%%%%%%%%%%%%%%
\section{Coherent States and  Berezin symbols}
\label{S:coherent}
%\input coherent.tex
%%%%%%%%%%%%%%%%%%%%%%%%%%%%%%%
%%%%%%%%%%   Section Coherent States and Symbols
%%      28.12.09/29.12./30.12.
%%%%%%%%%%%%%%%%%%%%%
\subsection{Coherent States}\label{SS:coherent}

%%%%%%%%%%%%%%%%%%%%%%%%%%%%%%%%%%%%%%%%%%%%%%%%%%%
Let the situation be as in the previous section. In particular $L$ is assumed
 to be
already very ample, $U=L^*$ is the dual of the quantum line bundle, 
$Q\subset U$ the unit circle bundle, and
$\tau:Q\to M$ the projection.
In particular, recall  the correspondence \refE{secident} 
$\ \psi_s(\alpha)=\alpha^{\otimes m}(s\tau(\a))$
of $m$-homogeneous functions $\psi_s$ on $U$ with sections 
of $L^m$.
To obtain this  correspondence we fixed the section $s$ and varied $a$.

Now we do the opposite. We fix
$\a\in U\setminus 0$ and vary the sections $s$. 
Obviously this yields 
a linear form
on $\ghm$ and hence with the help of the scalar product \refE{skp}
we  make the following
\begin{definition}\label{D:cohvec}
(a) The {\it coherent vector (of level m)}  associated to
the point $\a\in U\setminus 0$ is the unique element $\eam$
of $\ghm$ such that
\begin{equation}\label{E:cohvec}
\skp{\eam}{s}=\psi_s(\alpha)=\alpha^{\otimes m}(s(\tau(\alpha)))
\end{equation}
for all $s\in\ghm$.
\newline
(b) The {\it coherent state (of level m)}  associated to
$x\in M$ is  the projective class
\begin{equation} 
\e^{(m)}_x:= [\eam]\in\P(\ghm),\qquad \a\in\tau^{-1}(x), \a\ne 0.
\end{equation}
\end{definition}
Of course, we have to show that the object in (b) is well-defined.
Recall that $\skp{.}{.}$ denotes the scalar product on the space of
global sections $\gulm$.
In the convention of this review it will
be anti-linear in the first argument and linear in
the second argument.
The coherent vectors are antiholomorphic in $\alpha$ and fulfil
\begin{equation}\label{E:cohtrans}
e_{c\alpha}^{(m)}={\bar c}^m\cdot \eam,\qquad c\in\C^*:=\C\setminus\{0\}\ .
\end{equation}
Note that $\eam\equiv 0$ would imply 
that all sections will vanish at the point $x=\tau(\alpha)$.
Hence, the sections of $L$ cannot be used to embed $M$ into 
projective space, which 
is a contradiction to the very-ampleness of $L$.
Hence, $\eam\not\equiv 0$ and due to \refE{cohtrans} the 
class 
$$[\eam]:=\{s\in\ghm\mid \exists c\in\C^*:s=c\cdot \eam\}$$
is a well-defined 
element of the projective space $\P(\ghm)$, only depending on 
$x=\tau(\alpha)\in M$.

\medskip
This kind of coherent states go back to  
Berezin. A coordinate independent version and extensions
to line bundles were given by  
Rawnsley \cite{Raw}.
It plays an important role in the work of Cahen, Gutt, and
Rawnsley on the quantization of K\"ahler manifolds 
\cite{CGR1,CGR2,CGR3,CGR4},
 via Berezin's covariant symbols.
I will return to this in \refSS{berstar}.
In these works 
the coherent vectors are parameterized by the elements of
$L\setminus 0$.
The  definition here uses the points of the total space of the 
dual bundle $U$. It  has the advantage that one can consider
all tensor powers of $L$ together on an equal footing.
%%%%%%%%%%%%%%%%%%%%%%%%%%%%%%%%%%%%%%
\begin{definition}
The {\em coherent state embedding} is the 
antiholomorphic embedding
\begin{equation}\label{E:cohemb}
M\quad \to\quad \P(\ghm)\ \cong\ \pnc[N],
\qquad 
x\mapsto [e^{(m)}_{\tau^{-1}(x)}].
\end{equation}
\end{definition}
\noindent
Here $N=\dim\ghm -1$.
In this review, in abuse of notation, we will understand
under $\tau^{-1}(x)$ always a non-zero element of the 
fiber over $x$.
The coherent state embedding is up to conjugation the 
embedding of \refSS{embedd}
with respect to  an orthonormal basis of the sections.
In \cite{BerSchlcse}  further results on the geometry 
of the coherent state embedding are given.
%%%%%%%%%%%%%%%%%%%%%%%%%%%%%%%%%%%%%%%%%%%%%%%%%%%%%%%%%%

\subsection{Covariant Berezin  symbols}\label{SS:symbols}
%%%%%%%%%%%%%%%%%%%%%
We start with the 
\begin{definition}
The {\em covariant Berezin symbol $\sigma^{(m)}(A)$  (of level $m$)}
of an operator $A\in\eghm$ is defined as
\begin{equation}\label{E:covB}
\sigma^{(m)}(A):M\to\C,\qquad
x\mapsto \sm(A)(x):=
\frac {\skp {\eam}{A\eam}}{\skp {\eam}{\eam}},\quad
\alpha\in\tau^{-1}(x).
\end{equation}
\end{definition}
As the factors appearing in  \refE{cohtrans} will cancel, 
it is a well-defined function on $M$.
If the level $m$ is clear from the context I will sometimes drop
it in the notation.

We consider also  
the \emph{coherent projectors} used by Rawnsley
\begin{equation}\label{E:cohproj}
P^{(m)}_{x}=\frac {|\eam\rangle\langle \eam|}{\langle
  \eam,\eam\rangle},\qquad \alpha\in\tau^{-1}(x)
\ .
\end{equation}
Here we used the convenient bra-ket notation of the physicists.
Recall, if $s$ is a section then 
$$
P_x^{(m)}s=
\frac {|\eam\rangle \skp {\eam}{s}}{\langle \eam,\eam\rangle}\
=
\frac {\skp {\eam}{s}}{\skp {\eam}{\eam}} \eam.
$$

Again the projector is well-defined on $M$.
With its help the covariant symbol can be expressed as
\begin{equation}\label{E:covB2}
\sm(A)=\Tr(AP^{(m)}_x).
\end{equation}
{}From the definition of the symbol it follows 
that $\sm(A)$  is real analytic and 
\begin{equation}\label{E:covad}
\sm(A^*)=\overline{\sm(A)} \ .
\end{equation}

%%%%%%%%%%%%%%%%%%%%%%%%%%%%%%%%%%%%%%%%%%%%%%%%%%%%%%%%%%

\subsection{Rawnsley's $\epsilon$ function}\label{SS:epsilon}
%%%%%%%%%%%%%%%%%%%%%
Rawnsley \cite{Raw} introduced a very helpful function on the 
manifold $M$ relating the local metric in the bundle 
with the scalar product on coherent states.
In our dual description we define it in the following way.
\begin{definition}\label{D:eps}
{\it Rawnsley's  epsilon function} is the function
\begin{equation}\label{E:eps}
M\to\cim,\qquad 
x\mapsto \epsm(x):=\frac
{h^{(m)}(\eam,\eam)(x)}
{\skp {\eam}{\eam}},\quad \a\in\tau^{-1}(x).
\end{equation}
\end{definition}
With \refE{cohtrans} it is clear that it is a well-defined function
on $M$. Furthermore, using \refE{cohvec}
$$
0\ne \skp {\eam}{\eam}=\a^{\otimes m}(\eam(\tau(\a)))
$$
it follows that 
\begin{equation}
\label{E:nonzero}
\eam(x)\ne 0,\quad \text{for}\ x=\tau(\a), \quad\text{and}\ 
\epsm>0.
\end{equation}
Hence, we can define the modified measure
\begin{equation}\label{E:ome}
\Ome(x):=\epsm(x)\Om(x)
\end{equation}
for the space of functions on $M$ and obtain a modified
scalar product 
${\langle .,.\rangle}^{(m)}_\epsilon$ for $\cim$. 
\begin{proposition}
For $s_1,s_2\in\ghm$ we have
\begin{equation}\label{E:heps}
\begin{aligned}
h^{(m)}(s_1,s_2)(x)
&=\frac {\overline{\skp {\eam}{s_1}}\; 
   \skp {\eam}{s_2}}
{\skp {\eam}{\eam}}\cdot \epsm(x)
\\
&=\skp {s_1}{P^{(m)}_x s_2}\cdot
\epsm(x)\ .
\end{aligned}
\end{equation}
\end{proposition}
\begin{proof}
Due to \refE{nonzero} we can represent every section $s$ locally
at $x$ as $s(x)=\hat s(x)\eam$ with a local function 
$\hat s$.
Now
$$
\skp {\eam}{s}=
\a^{(m)}(\hat s(x)\eam(x))=
\hat s(x)\a^{(m)}(\eam(x))=
\hat s(x)\skp {\eam}{\eam}.
$$
We rewrite 
$h^{(m)}(s_1,s_2)(x)=\overline{\hat s_1}{s_2}
\hm(\eam,\eam)(x)$, and obtain
\begin{equation}
\hm(s_1,s_2)(x)=
\frac {\overline{\skp {\eam}{s_1}}}
{\skp {\eam}{\eam}}
\frac {{\skp {\eam}{s_2}}}
{\skp {\eam}{\eam}}\cdot
\hm(\eam,\eam)(x).
\end{equation}
{}From the definition \refE{eps} the first relation follows.
Obviously, it can be rewritten with the coherent
projector to obtain the second relation.
\end{proof}

There exists another useful description of the epsilon
function.
\begin{proposition}\label{P:epssec}
Let $s_1,s_2,\ldots ,s_k$ be an arbitrary orthonormal
basis of $\ghm$. Then
\begin{equation}
\epsm(x)=\sum_{j=1}^k \hm(s_j,s_j)(x).
\end{equation}
\end{proposition}
\begin{proof}
For every vector $\psi$ in a finite-dimensional hermitian
vector space with orthonormal basis ${s_j}, j=1,\ldots ,k$ the 
coefficient with respect to the basis element $s_j$ is given 
by $\psi_j=\skp {s_j}{\psi}$.
Furthermore, 
$
\skp {\psi}{\psi}=||\psi||^2=\sum_j\overline{\psi_j}\psi_j
$.
Now from \refE{heps}
$$
\sum_{j=1}^k h(s_j,s_j)(x)=
\frac {\epsm(x)}{\skp {\eam}{\eam}}
\sum_{j=1}^k \overline{\skp {\eam}{s_j}}\skp {\eam}{s_j}\ .
$$
Hence the claim.
\end{proof}

In certain special cases the functions $\epsm$ will be
constant as function of  the points of the manifold.
In this case we can apply \refP{Atr} below 
for $A=id$, the identity operator, and
obtain
\begin{equation}
\epsm =\frac {\dim\ghm}{\volu M}.
\end{equation}
Here $\volu M$ denotes the volume of the manifold with 
respect to the Liouville measure.
Now the question  arises when 
$\epsm$ will be constant, resp. when 
the measure $\Ome$ will be the standard
measure (up to a scalar).
From \refP{epssec} it follows that 
if there is a transitive group action on the manifold and
everything, e.g. K\"ahler form, bundle, metric, 
is homogenous with respect to the action this will be the case.
An example is given by 
$M=\pnc$. 
By a result of Rawnsley \cite{Raw},
 resp. Cahen, Gutt and
Rawnsley \cite{CGR1},
$\epsm\equiv const$ if and only if the quantization  is 
projectively induced. 
This means that under  the conjugate of the coherent state embedding, 
the K\"ahler form $\w$ of $M$ coincides with the pull-back of the
Fubini-Study form. Note that in general this is not the
case, see \refSS{pfsm}.

%%%%%%%%%%%%%%%%%%%%%%%%%%%%%%%%%%%%%%%%%%%%%%%%%%%%%%%%%%

\subsection{Contravariant Berezin  symbols}
\label{SS:contra}
%%%%%%%%%%%%%%%%%%%%%
Recall the modified Liouville measure 
\refE{ome} and modified scalar product for
the functions on $M$ introduced in the last subsection.
\begin{definition}
Given an operator $A\in\eghm$ then a 
 {\em contravariant Berezin symbol} $\svm(A)\in\cim$
of $A$  is defined by the representation of the operator
$A$ as integral
\begin{equation}\label{E:contra}
A=\int_M\svm(A)(x)P^{(m)}_x\,\Ome(x),
\end{equation}
if such a representation exists.
\end{definition}
\begin{proposition}\label{P:contoe}
The Toeplitz operator $\Tfm$ admits a representation \refE{contra} with
\begin{equation}\label{E:contoe}
\svm(\Tfm)=f\ , 
\end{equation}
i.e. the function $f$ is a contravariant symbol of the Toeplitz operator
$\Tfm$.
Moreover, every operator $A\in\eghm$ has a contravariant symbol.
\end{proposition}
\begin{proof}
Let $f\in\cim$ and set
\begin{equation}
A:=\int_M f(x)P^{(m)}_x\,\Ome(x),
\end{equation}
then $\svm(A)=f$.
For arbitrary $s_1,s_2\in\ghm$ we calculate
(using \refE{heps})
\begin{multline}
\skp {s_1}{As_2}=\int_M f(x)\skp {s_1}{P^{(m)}_xs_2}\,\Ome(x)
=\int_M f(x)\hm(s_1,s_2)(x)\,\Om(x)\\
=\int_M \hm(s_1,fs_2)(x)\Om(x)=
\skp {s_1}{fs_2}=\skp {s_1}{\Tfm s_2}\ .
\end{multline}
Hence $\Tfm=A$.
As the Toeplitz map is surjective (\refP{sur}) every operator is 
a Toeplitz operator, hence has a contravariant symbol.
\end{proof}
Note that given an operator its  contravariant symbol on a fixed
level $m$ is
 not uniquely defined.

\medskip

 We introduce on $\eghm$ the Hilbert-Schmidt norm 
\begin{equation}\label{E:HS}
\skps {A}{C}{HS}=Tr(A^*\cdot C)\ .
\end{equation}
\begin{theorem}\label{T:adj}
The Toeplitz map $f\to \Tfm$ and the covariant symbol map
$A\to\sm(A)$ are adjoint:
\begin{equation}\label{E:adjoint}
\skps {A}{\Tfm}{HS}=\skpsm {\sm(A)} {f} {\epsilon} \ .
\end{equation}
\end{theorem}
\begin{proof}
\begin{equation}
\langle A,\Tfm\rangle =Tr( A^*\cdot \Tfm)
=Tr( A^*\int_Mf(x)P^{(m)}_x\,\Ome(x))
=\int_Mf(x) Tr(A^*\cdot P^{(m)}_x)\Ome(x).
\end{equation}
Now applying the definition \refE{covB2} and equation \refE{covad}
\begin{equation}
\langle A,\Tfm\rangle =\int_M f(x)\sm( A^*)\Ome(x)
=\int_M \overline{\sm(A)}(x)f(x)\Ome(x)
={\langle \sm(A),f(x)\rangle} _\epsilon^{(m)}.
\end{equation}
\end{proof}
As every operator has a contravariant symbol we can also conclude
\begin{equation}\label{E:cocont}
\skps {A}{B}{HS}=\skpsm {\sm(A)} {\svm(B)} {\epsilon}. 
\end{equation}
From \refT{adj} by  using the 
surjectivity of the Toeplitz map 
we get
\begin{proposition}\label{P:syminj}
The covariant symbol map $\sm$ is injective.
\end{proposition} 
Another application is the following
\begin{proposition}\label{P:Atr}
\begin{equation}\label{E:Atr}
\Tr A=\int_M\sm(A)\,\Ome.
\end{equation}
\end{proposition}
\begin{proof}
We use $Id=T_1$ and by \refE{adjoint}
$
\Tr A=\skps {A}{Id}{HS}=
\skpsm {\sm(A)}{1}{\epsilon}
$.
\end{proof}

%%%%%%%%%%%%%%%%%%%%%%%%%%%%%%%%%%%%%%%%%%%%%%%%%%%%%%%%%%

\subsection{Berezin star product}\label{SS:berstar}
%%%%%%%%%%%%%%%%%%%%%
Under certain very restrictive conditions Berezin covariant
symbols can be used to construct a star product, called the
{\it Berezin star product}.
Recall that \refP{syminj} says that 
the linear symbol map
\begin{equation}\label{E:symmap}
\sm:\eghm\to \cim
\end{equation}
is injective.
Its image is a subspace $\A^{(m)}$ of $\Cim$, 
called the subspace of covariant symbols of level $m$.
If $\sm(A)$ and $\sm(B)$  are elements of this subspace the
operators $A$ and $B$ will be uniquely fixed.
Hence also $\sm(A\cdot B)$. Now one takes
\begin{equation}
\sm(A)\star_{(m)}\sm(B):=\sm(A\cdot B)
\end{equation}
as definition for an  associative and noncommutative
product $\star_{(m)}$ on $\A^{(m)}$.

It is even possible to give an analytic expression for
the resulting symbol.
For this we introduce 
the {\it two-point function}
\begin{equation}\label{E:two}
\psi^{(m)}(x,y)=
\frac {\langle \eam, \ebm\rangle \langle \ebm,\eam
\rangle }{\langle \eam,\eam\rangle \langle \ebm,\ebm \rangle }
\end{equation}
with $\a=\tau^{-1}(x)=x$ and $\beta=\tau^{-1}(y)$.
This function is well-defined on $M\times M$.
Furthermore, we have the {\it two-point symbol}
\begin{equation}\label{E:twos}
\sm(A)(x,y)
=\frac {\langle \eam,A \ebm \rangle }
{\langle \eam, \ebm \rangle }.
\end{equation}
It is the analytic extension of the real-analytic covariant symbol.
It is well-defined on an open dense subset of $M\times M$
containing
the diagonal.
Using \refE{heps} we  express 
\begin{multline}
\sm(A\cdot B)(x)=
\frac {\langle \eam,A\cdot B\, \eam\rangle }
{\langle \eam,\eam\rangle }=
\frac {\langle A^*\eam, B\, \eam\rangle }
{\langle \eam,\eam\rangle }
\\
=\int_M\hm(A^*\eam,B\eam)(y)
\frac {\Omega(y)}{\langle \eam,\eam\rangle }
=
\int_M
\frac {
\langle \eam,A\ebm\rangle \langle \ebm,B\eam\rangle}
{ \langle \ebm,\ebm\rangle }
\frac {\epsm(y)\Omega(y)}{\langle \eam,\eam\rangle }
\\
=
\int_M \sm(A)(x,y)\sm(B)(y,x)\cdot \psi^{(m)}(x,y)\cdot
\epsm(y)\Omega(y)\ .
\end{multline}

The crucial problem is how to relate different levels $m$ to define 
for all possible symbols a unique product not depending on $m$.
In certain special situations like these studied by 
Berezin himself \cite{Beress} and Cahen, Gutt, and Rawnsley
\cite{CGR1} 
the subspaces are nested into each other and the union
$\A=\bigcup_{m\in\N}\A^{(m)}$ is a dense subalgebra of $\Cim$.
Indeed, in the cases 
considered,  the manifold is a homogenous manifold and 
the epsilon function $\epsm$ is 
a constant.
A detailed analysis shows that in this case  a
star product is given.

For further examples, for which  this method works (not necessarily 
compact) see other articles by   Cahen, Gutt, and Rawnsley
\cite{CGR2,CGR3,CGR4}.
For related results see also work of Moreno and Ortega-Navarro 
\cite{MoOr}, \cite{Mor1}.
In particular,  also the work of Engli\v s
\cite{Engbk,Eng2,Eng1,Eng0}. 
Reshetikhin and Takhtajan \cite{ResTak} gave a construction
of a (formal) star product using formal integrals in 
the spirit of the Berezin's covariant symbol  construction.

%%%%%%%%%%%%%%%%%%%%%%%%%%%%%%%%%%%%%%%%%%%%%%%%%%%%%%%
\section{Berezin  transform}
\label{S:btrans}
%\input btrans.tex
%%%%%%%%%%%%%%%%%%%%%%%%%%%%%%%%%%%%
%%%%%%%% The Berezin transform
%%%%%%%%%%%%%%%%%%%%%%%%%%%%%%%%%%
%%%%         25.12.09/29.12./30.12.
%%%%%%%%%%%%%%%%%%%%%%%%%%%%%%%%%

\subsection{The definition}\label{SS:btdef}
%%%%%%%%%%%%%%%%%%%%%%%%%%%%%%%%%%%%%%%%%%%%%%%%%
Starting from $f\in\cim$ we can assign to it its Toeplitz operator
$\Tfm\in\eghm$ and then assign to $\Tfm$ the covariant symbol 
$\sm(\Tfm)$. It 
is again an element of $\cim$.  
\begin{definition}
The map 
\begin{equation}
\Cim\to\Cim,\qquad
f\mapsto
\Im(f):=\sm(\Tfm)
\end{equation}
is called the {\it Berezin transform (of level $m$)}.
\end{definition}

{}From the point of view of Berezin's approach 
the operator $T_f^{(m)}$ has as a  contravariant symbol $f$.
Hence $\Im$ gives a correspondence between contravariant symbols
and covariant symbols of operators.
The Berezin transform was introduced and studied by 
Berezin \cite{Beress} for certain classical symmetric 
domains in $\C^n$. These results where 
extended by Unterberger and Upmeier \cite{UnUp},
see also Engli\v s \cite{Eng1,Eng2,Engbk} 
and Engli\v s and Peetre \cite{EnPe}.
Obviously, the Berezin transform makes perfect  sense
in the compact K\"ahler case which we
consider here.
%%%%%%%%%%%%%%%%%%%%%%%%%%%%%%%%%%%%%%%%%%%%%%%%%%%%%%%%%%%%

\subsection{The asymptotic expansion}\label{SS:asymp}
The results presented here  are
joint work with Alexander Karabegov \cite{KS}.
See also \cite{Schlbol} for an overview.
\begin{theorem}\label{T:btapp}
Given $x\in M$ then the Berezin transform 
$\Im(f)$ evaluated at the point $x$ has a complete asymptotic 
expansion in powers of $1/m$ as $m\to\infty$
\begin{equation}\label{E:btapp}
\Im(f)(x)\quad\sim \quad\sum_{i=0}^\infty I_i(f)(x)\frac {1}{m^i}, 
\end{equation}
where   $I_i:\cim\to\cim$ are maps with
\begin{equation}
I_0(f)=f,\qquad I_1(f)=\Delta f.
\end{equation} 
\end{theorem}
Here the $\Delta$ is the  usual Laplacian with respect
to the metric given by the  K\"ahler form $\w$.

Complete asymptotic expansion means the following.
Given $f\in\Cim$, $x\in M$ and an $r\in\N$ then there
exists  a positive constant $A$ such that
\begin{equation}
{\left|\Im(f)(x)- \sum_{i=0}^{r-1} I_i(f)(x)\frac {1}{m^i}\right|}_{\infty}
\quad\le\quad \frac {A}{m^r}\ .
\end{equation}
In \refSS{berg}
I will give some remarks on the proof.
But before I present  you a nice application
%%%%%%%%%%%%%%%%%%%%%%%%%%%%%%%%%%%%%%%%%%%%%%%%%%%%%%%%%%%%%

\subsection{Norm preservation of the BT operators}\label{SS:normp}
%%%%%%%%%%%%%%%%%%%%%%%%
In \cite{Schlbia98} I conjectured \refE{btapp} (which is now a
mathematical
result) 
and showed how such an asymptotic expansion supplies a different
proof  of \refT{approx}, part (a).
For completeness I  reproduce the proof here.
\begin{proposition}
\begin{equation}\label{E:symine}
|\Im(f)|_\infty=|\sm(\Tfm)|_\infty\quad\le \quad||\Tfm||\quad\le\quad   
|f|_\infty\ .
\end{equation}
\end{proposition}
\begin{proof}
Using Cauchy-Schwarz inequality we calculate ($x=\tau(\alpha)$)
\begin{equation}
| \sm(\Tfm)(x)|^2=
\frac {|\skp {\eam}{\Tfm\eam}|^2}{{\skp {\eam}{\eam}}^2}\le
\frac {\skp {\Tfm\eam}{\Tfm\eam}}{\skp {\eam}{\eam}}\le
||\Tfm||^2\ .
\end{equation}
Here the last inequality  follows from the definition of the operator norm.
This shows the first inequality in \refE{symine}.
For the second inequality introduce the multiplication
operator $M_f^{(m)}$ on $\gulm$. Then 
$\ ||\Tfm||=||\Pim\,M_f^{(m)}\,\Pim||\le ||M_f^{(m)}||\ $ and
for  $\varphi\in\gulm$,  $\varphi\ne 0$
\begin{equation}
\frac {{||M_f^{(m)} \varphi||}^2}{||\varphi||^2}=
\frac {\int_M h^{(m)}(f \varphi,f \varphi)\Omega}
 {\int_M h^{(m)}(\varphi,\varphi)\Omega}
=
\frac {\int_M f(z)\overline{f(z)}h^{(m)}(\varphi,\varphi)\Omega}
{\int_M h^{(m)}(\varphi,\varphi)\Omega}
\le
|f|{}_\infty^2\ .  
\end{equation}
Hence,
\begin{equation}
||\Tfm||\le ||M_f^{(m)}||=\sup_{\varphi\ne 0}
\frac {||M_f^{(m)}\varphi||}{||\varphi||}\le |f|_\infty .
\end{equation}
\end{proof}
\begin{proof} (\refT{approx} part (a).)
Choose as $x_e\in M$ a point with $|f(x_e)|=|f|_\infty$.
From the fact that the  Berezin
transform has
as leading term the identity   it follows that
$\ |(I^{(m)}f)(x_e)-f(x_e)|\le C/m\ $ with a suitable constant
$C$.
Hence,
$\ \left| |f(x_e)|-|(I^{(m)}f)(x_e)| \right| \le C/m\ $
and 
\begin{equation}\label{E:absch}
|f|_\infty-\frac Cm=|f(x_e)|-\frac Cm\quad\le\quad
|(I^{(m)}f)(x_e)|\quad\le\quad |I^{(m)}f|_\infty\ .
\end{equation}
Putting \refE{symine} and \refE{absch} together we obtain
\begin{equation}\label{E:thma}
|f|_\infty-\frac Cm\quad\le\quad ||T_f^{(m)}||\quad\le\quad
|f|_\infty\ .
\end{equation}
\end{proof}

%%%%%%%%%%%%%%%%%%%%%%%%%%%%%%%%%%%%%%%%%%
\subsection{Bergman kernel}
\label{SS:berg}
%%%%%%%%%%%%%%%%%%%%%%%%%%%%%%%
To understand the Berezin transform better we have to study the
Bergman kernel.
Recall from \refS{global}
the Szeg\"o projectors 
$ \Pi:\Lqv\to \Hc$ and its components 
$\pimh:\Lqv\to\Hm$, the Bergman projectors.
The Bergman projectors have smooth integral kernels,
the {\it Bergman kernels} 
$\Bm(\alpha,\beta)$ defined on $Q\times Q$, i.e.
\begin{equation}
\pimh(\psi)(\alpha)=\int_Q\Bm(\alpha,\beta)\psi(\beta)\mu(\beta).
\end{equation}
The Bergman kernels can be expressed with the help of the
coherent vectors.
\begin{proposition}\label{P:kernel}
\begin{equation}
\Bm(\alpha,\beta)=\psi_{\ebm}(\alpha)=
\overline{\psi_{\eam}(\beta)}=\skp{\eam}{\ebm}.
\end{equation}
\end{proposition}
\noindent
For the proofs of this and the following propositions see
\cite{KS}, or \cite{Schlbol}.

Let $x,y\in M$ and choose $\alpha,\beta\in Q$ with
$\tau(\alpha)=x$ and $\tau(\beta)=y$ then the functions
\begin{equation}\label{E:um}
u_m(x):=\Bm(\alpha,\alpha)=
\skp{\eam}{\eam},
\end{equation}
\begin{equation}\label{E:vm}
v_m(x,y):=\Bm(\alpha,\beta)\cdot \Bm(\beta,\alpha)=
 \skp{\eam}{\ebm}\cdot \skp{\ebm}{\eam}
\end{equation}
are well-defined on $M$ and on $M\times M$ respectively.
The following proposition gives an integral representation of the
Berezin transform.
\begin{proposition}\label{P:kernelint}
\begin{equation}\label{E:kernelint}
\begin{aligned}
\left(\Im(f)\right)(x)&=\frac 1{\Bm(\alpha,\alpha)}
\int_Q \Bm(\alpha,\beta)\Bm(\beta,\alpha)\tau^*f(\beta)\mu(\beta)
\\
&=
\frac 1{u_m(x)}
\int_M v_m(x,y)f(y)\Omega(y)\ .
\end{aligned}
\end{equation}
\end{proposition}
Typically, asymptotic expansions can be obtained using 
stationary phase integrals. But for such an asymptotic expansion
of the integral representation of the Berezin transform we
will not only need an asymptotic expansion of the Bergman kernel
along the diagonal (which is well-known) but in a
neighbourhood of it.
This is one of the key results obtained in \cite{KS}.
It is based on works of Boutet de Monvel and Sj\"ostrand 
\cite{BS} on the Szeg\"o kernel and in generalization of a result of
Zelditch \cite{Zel} on the Bergman kernel on the diagonal.
The integral representation 
is used then to prove
the existence  of the 
asymptotic expansion of the Berezin transform.

Having such an asymptotic expansion it still remains to
identify its terms.
As it was explained in \refSS{kara}, 
Karabegov assigns to every
formal 
deformation quantizations 
with the ``separation of variables'' property 
a  {\em formal Berezin transform} $I$. 
In \cite{KS} it is shown that 
there is an explicitely specified 
star product $\ \star\ $ (see Theorem 5.9 in \cite{KS})
with associated formal Berezin transform such that 
if we replace 
$\frac 1m$ by the formal variable $\nu$ in the asymptotic  expansion of
the 
Berezin transform $\Im f(x)$ we obtain
$I(f)(x)$.
This finally proves \refT{btapp}.
We will exhibit the star product $\ \star\ $ in the next section.

%%%%%%%%%%%%%%%%%%%%%%%%%%%%%%%%%%%%%%%%%%%%%%%%%%%%%%%%%%%%%%%%%%
\subsection{Identification of the BT star product}\label{SS:ident}
%%%%%%%%%%%%%%%%%%%%%%%%%%%%%%%%%%%%%%%%%%%%%%%%%%%%

Moreover in \cite{KS} there is another object introduced, the
{\it twisted product}
\begin{equation}
R^{(m)}(f,g):=\sm(\Tfm\cdot \Tgm) \ .
\end{equation}
Also for it the 
existence of a complete asymptotic expansion was shown.
It was identified with a twisted formal product. 
This  allows the identification of the BT star product with 
a special star product  within the classification of
Karabegov. From this identification the properties
of \refT{bstarad} of
locality, separation of variables type, and the calculation to
the classifying forms and classes for the BT star product
follows.

As already announced in \refSS{kara},
the BT star product $\star_{BT}$ is the opposite of the dual star product
of a certain star product $\star$.
To identify $\star$ we will give its 
classifying Karabegov form $\widehat{\w}$ .
As already mentioned above, Zelditch \cite{Zel} proved that the
the function $u_m$ \refE{um} has a complete asymptotic expansion 
in powers of $1/m$. 
In detail he showed
\begin{equation}
u_m(x)\sim m^n\sum_{k=0}^\infty \frac 1{m^k}\,b_k(x), \quad
b_0=1.
\end{equation}
If we replace in the expansion 
$1/m$ by the formal variable $\nu$ we obtain a formal function
$s$
defined by
\begin{equation}
e^{s}(x)=
\sum_{k=0}^\infty \nu^k\,b_k(x).
\end{equation}
 Now take as formal potential  \refE{formp}
$$\widehat{\Phi}=\frac 1\nu \Phi_{-1}+s,
$$  
where $\Phi_{-1}$ is the local K\"ahler potential of the 
K\"ahler form $\w=\w_{-1}$.
Then $\widehat{\w}=\i\partial\bar\partial\widehat{\Phi}$.
It might be also written in the form
\begin{equation}\label{E:fk}
\widehat{\w}=\frac 1\nu \w+\F(\i\partial\bar\partial\log\Bm(\a,\a)).
\end{equation}
Here we denote the replacement of $1/m$ by 
the formal variable $\nu$ by the symbol $\F$.

%%%%%%%%%%%%%%%%%%%%%%%%%%%%%%%%%%%%%%%%%%%%%%%%%%%%%%%%%%%%%%%%%%
\subsection{Pullback of the Fubini-Study form}\label{SS:pfsm}
Starting from the  K\"ahler manifold $(M,\omega)$ and after choosing an
orthonormal basis of the space $\ghm$ we obtain an embedding 
$$\phi^{(m)}:M\to\P^{N(m)}$$ 
of $M$ into projective space of dimension $N(m)$.
On  $\P^{N(m)}$ we have the standard K\"ahler form, the Fubini-Study  form
$\omega_{FS}$. 
The pull-back
$(\phi^{(m)})^*\omega_{FS}$ will not depend on the orthogonal basis
chosen for the embedding.  But
in general it will not coincide with 
a scalar multiple of the K\"ahler form $\omega$ we started with
(see \cite{BerSchlcse} for a thorough discussion of the
situation).

It was shown by Zelditch \cite{Zel}, by generalizing a result
of Tian \cite{Tian}, that  
$(\Phi^{(m)})^*\omega_{FS}$ admits a complete asymptotic expansion in 
powers of $\frac 1m$ as $m\to\infty$.
In fact it is related to the asymptotic expansion of the
Bergman kernel  \refE{um} along the diagonal. 
The pull-back can be given as
\cite[Prop.9]{Zel}
\begin{equation}
\left(\phi^{(m)}\right)^*\w_{FS}=m\w+\i\partial\bar\partial\log u_m(x)\ .
\end{equation}
If we again replace $1/m$ by $\nu$ we obtain via \refE{fk}
the Karabegov form introduced in \refSS{kara}
\begin{equation}
\widehat{\w} =\F(\left(\phi^{(m)}\right)^*\w_{FS}).
\end{equation}

%%%%%%%%%%%%%%%%%%%%%%%%%%%%%%%%%%%%%%%%%%%%%%%%%%%%%%%
%\section{Additional results, applications, and open ends}
%\label{S:add}
%\input add.tex

%%%%%%%%%%  bibliography  %%%%%%%%%%
%\bibliographystyle{amsplain}
%\bibliography{database}

%\end{thebibliography}
%\input berrev.bbl
%\providecommand{\bysame}{\leavevmode\hbox to3em{\hrulefill}\thinspace}

%%%%%%%%%%%%%%%%%%%%%%%%%%%%%%%%%%%%%%%%%%%%%%%%%%%%
\end{document}